\documentclass{amsart}

\newtheorem{theorem}{Theorem}[section]

\theoremstyle{definition}
\newtheorem{definition}[theorem]{Definition}
\newtheorem{ex}[theorem]{Example}

\theoremstyle{remark}
\newtheorem{remark}[theorem]{Remark}

\numberwithin{equation}{section}

\newtheorem {Theorem}{Theorem} 
\newtheorem {thm}[Theorem]{Theorem} 

\newtheorem {prop}[Theorem]{Proposition}
\newtheorem {cor}[Theorem]{Corollary}
\newtheorem {lem}[Theorem]{Lemma}

\newtheorem{alg}[Theorem]{Algorithm}

\usepackage{amssymb, amsfonts, amsmath}
\usepackage{amsthm}
\usepackage{hyperref} 
\usepackage{mathrsfs}

\newcommand{\sref}[1]{(\ref{#1})}

\newcommand{\bmat}{\left(\begin{array}}
\newcommand{\emat}{\end{array}\right)}
\newcommand{\beq}{\begin{equation}}
\newcommand{\eeq}{\end{equation}}

\newcommand{\Mod}[1]{\,(\mbox{\rm mod}~#1)}
\newcommand{\C}{\mathbb{C}}

\newcommand{\Q}{\mathbb{Q}}
\newcommand{\Z}{\mathbb{Z}}

\newcommand{\Aut}{\operatorname{Aut}}

\newcommand{\cont}{\operatorname{cont}}
\newcommand{\disc}{\operatorname{disc}}
\newcommand{\Div}{\operatorname{Div}}
\newcommand{\End}{\operatorname{End}}
\newcommand{\gen}{\operatorname{gen}}
\newcommand{\GL}{\operatorname{GL}}
\newcommand{\Hom}{\operatorname{Hom}}

\newcommand{\lcm}{\operatorname{lcm}}
\newcommand{\NS}{\operatorname{NS}}

\newcommand{\SL}{\operatorname{SL}}

\usepackage{calc}
\usepackage{accents}

\newcommand{\DistTo}{\xrightarrow{
   \,\smash{\raisebox{-0.65ex}{\ensuremath{\scriptstyle\sim}}}\,}}

 \newcommand{\charec}{\operatorname{char}}

      \newcommand{\odd}{\operatorname{odd}}
     \newcommand{\ev}{\operatorname{ev}}
     \newcommand{\adj}{\operatorname{adj}}
  
  \newcommand{\Gal}{\operatorname{Gal}}
    \newcommand{\rad}{\operatorname{rad}}

\newcommand{\sps}{\vspace{3pt}}   
\newcommand{\spm}{\vspace{6pt}}

\begin{document}

\title{Constructing Genus 2 Curves with Given Refined Humbert Invariants}

\author{Harun KIR}
\address{}
\curraddr{}
\email{harun.kir@ens-lyon.fr}
\thanks{The author was supported by the Agence
Nationale de la Recherche under grants ANR-22-PETQ-0008 (PQ-TLS)}

\subjclass[2010]{Primary}

\keywords{Elliptic curves with complex multiplication; product abelian surfaces; genus
2 curves; the refined Humbert invariant; Humbert surfaces; integral binary and ternary quadratic
forms and their representations.}

\date{}

\dedicatory{}

\begin{abstract}
In 1994, Kani introduced an algebraic version of the Humbert invariant, known as the \textit{refined Humbert invariant}. This invariant $q_C$ is a positive definite quadratic form attached to a smooth curve $C$ of genus $2$. It serves as a vital tool, as many geometric properties of $C$ are reflected in the arithmetic properties of $q_C$. When the Jacobian $J_C$ of a genus $2$ curve $C$ is isogenous to a product of an elliptic curve with complex multiplication, the forms $q_C$ have been completely classified recently. In this paper, building upon this classification, we present a constructive algorithm that produces $J_C$ and \textit{a divisorial representative} of a curve $C$ of genus $2$ such that its refined Humbert invariant $q_C$ is equivalent to a given integral ternary quadratic form.

\end{abstract}

\maketitle

\section{Introduction}

The study of (smooth) genus $2$ curves and the properties that distinguish them is an active and compelling area of research. Various approaches exist depending on the specific focus; following Kani, one particularly fruitful direction involves a special integral quadratic form $q_C$ intrinsically attached to a genus $2$ curve $C$. This form, known as the \textit{refined Humbert invariant} of $C$, serves as a powerful tool, translating geometric problems into arithmetic ones and providing deep insights into the nature of these curves.

Before detailing the contributions of this paper, let us highlight the utility of this invariant. The refined Humbert invariant allows for the description of genus $2$ curves with a prescribed group of automorphisms (see, e.g., \cite{ESCII, Automorphism, cas}) or with additional structures, such as having an elliptic subgroup of a prescribed degree (see, e.g., \cite{kani1994elliptic, ESCI, SubcoversofCurves}). Furthermore, these forms facilitate the understanding of CM points in the intersection of Humbert surfaces (see, e.g., \cite{harun_1, kir2024curvesPHD}). Finally, it is possible to determine the bad reductions of a genus $2$ curve solely by computing these forms; see \cite{GRV}.

The utility of the refined Humbert invariant extends well beyond these specific instances. It provides elegant solutions to compelling geometric problems concerning genus $2$ curves, as demonstrated in \cite{kani2014jacobians, MJ}. Additionally, its relevance extends to cryptographic contexts, as discussed in \cite{Eda_degree_analysis}.

The primary aim of this paper is to present an algorithm that constructs a genus $2$ curve $C$ such that its refined Humbert invariant $q_C$ is equivalent to (as a quadratic form) a given ternary quadratic form $q$ (possessing prescribed necessary properties). This yields a method for constructing explicit genus $2$ curves with specific features, exploiting the correspondence between the arithmetic properties of $q_C$ and the geometric properties of $C$.

Recall that Humbert \cite{Humbert} (1899) introduced an analytic invariant, now known simply as the \textit{Humbert invariant}, to classify abelian surfaces $A/\C$ with $\End(A)\neq \Z$. This invariant is an integer. The \textit{refined Humbert invariant} is its algebraic generalization introduced by Kani \cite{kani1994elliptic} in 1994. It is a positive definite quadratic form $q_{(A,\theta)}$ attached to a principally polarized abelian surface $(A,\theta)/K$ over an algebraically closed field $K$.

In the specific case where $(A, \theta)=(J_C,\theta_C)$ for a curve $C$ of genus $2$ (where $J_C$ is the Jacobian of $C$ and $\theta_C$ is the theta divisor), we denote the refined Humbert invariant by $q_C:=q_{(J_C,\theta_C)}$.

 The problem of determining which ternary quadratic forms $f$ are equivalent  to a refined Humbert invariant $q_{(A,\theta)}$ for some principally polarized abelian surface $(A,\theta)/K$ was completely solved in \cite{refhum} and \cite{harun_1}. We use the following terminology for such forms:
\begin{definition}
\label{d: definition_special_form}
     Let $f$ be a positive definite ternary quadratic form.
     If $f$ is equivalent to a refined Humbert invariant $q_{(A,\theta)}$, for some $(A,\theta)$, denoted by $f \sim q_{(A,\theta)}$, then we say that $f$ is a \textit{geometric form}.
\end{definition}

In this article, we provide an algorithm that, given a geometric form $f$, constructs a principally polarized abelian surface $(A, \theta)$ such that $f\sim q_{(A, \theta)}$. More precisely, we first
construct $A/\C$ in terms of lattices, as outlined in Proposition \ref{p:construction_A}, and then
identify $\theta$ on $A$ as the image of a divisor under an (explicit) isomorphism $\mathbf{D}$ defined in \sref{eq: isomorphism_NS(A)}. Thus, our construction of a divisorial representative of $\theta$ corresponds
to a triple containing two integers and one elliptic curve isogeny.

Drawing on the constructive nature of the existence proofs for geometric forms, we translate these theoretical results into an explicit algorithmic procedure. A key step in this procedure requires finding a prime number represented by a given positive definite primitive ternary quadratic form. While unconditional upper bounds for such primes exist, we assume the Generalized Riemann Hypothesis (GRH) to utilize a sharper bound (cf.\ Proposition \ref{prop: BW_existence_prime}). We emphasize, however, that the correctness of the algorithm's output does not depend on the GRH.

\begin{thm}[GRH]
\label{thm: main}
Assume GRH holds. For a given geometric form $f$, we construct $(A,\theta)/\C$ such that $f \sim q_{(A,\theta)}$ by following the steps of Algorithms \ref{m: method_1_primitive} and \ref{m: method_2_imprimitive}.
\end{thm}

Algorithms \ref{m: method_1_primitive} and \ref{m: method_2_imprimitive} constitute the core results of this paper, providing the explicit construction of $(A,\theta)$ as established in Theorem \ref{thm: main}. The structure of these algorithms reflects the constructive proof that demonstrates a ternary quadratic form satisfying the conditions of \sref{eq: class_ref_condition_1}--\sref{classificationconditions2} is a geometric form.  Furthermore, we present a simplified version of the algorithm that yields identical outputs for the same input. By omitting some terminologies (such as the genus of a binary form) required for the proofs, this formulation is significantly simpler and more suitable for implementation, as we have done in SageMath \cite{SageMath}\footnote{\url{https://github.com/harunkir/Constructing-Genus-2-Curves-with-Given-Refined-Humbert-Invariants}}

Kani \cite{refhum} outlined an algorithm (Algorithm 27 in \cite{refhum}) for constructing such an $(A,\theta)$ when $f$ is primitive. We refine Kani's approach to provide a fully explicit formulation in that case. This refinement entails extensive work involving both binary and ternary integral quadratic forms.

\begin{remark}
The construction algorithm was originally presented in Chapter 4 of the author's PhD thesis. However, that version was significantly more complex due to certain algorithmic difficulties, specifically the requirement to produce a prime number represented by given ternary and binary quadratic forms. Fortunately, recent crypto work by Wesolowski \cite{BW_hard_problems} addresses these obstructions quite well. Incorporating these results has considerably simplified the algorithms presented here compared to those in the thesis. Additionally, the requirement to analyze the local behavior of imprimitive ternary forms, which previously required significant effort, is now handled using the results of \cite{harun_1}.
\end{remark}

 Since the classification of geometric forms relies heavily on the classical theory of ternary and binary quadratic forms such as the results found in Chapters II and III of Dickson \cite{dicksonsbook}, we explicitly outline several algorithms by adapting the constructive proofs from classical literature rather than simply citing them (see Algorithms \ref{alg: find_repr_assigned_value} and \ref{alg: produce_phi_p}, along with Remark \ref{rem: binary_rep_assigned_value}). Furthermore, as we utilize results from the unpublished article \cite{refhum} (specifically when $f$ is a primitive geometric form), we provide dual citations to both that paper and the author's PhD thesis for clarity and convenience.

Regarding the construction algorithm itself: for a given geometric form $f$, we first explicitly construct a degree map $q_{E,E'}$ on $\Hom(E,E')$ such that $f\sim q_{(E\times E', \theta)}$ for some principally polarized abelian surface $(E\times E',\theta)$ (see \S\ref{S: section_constructionA}). This construction relies on the close relationship between the local behavior (assigned characters) of the degree map and the reciprocal form of $f$. Then, for a given positive definite binary quadratic form $q$, we employ results from \cite{kani2011products} (or cf.\ \cite{refhum})  and give a construction of $E$ and $E'$ such that $q\sim q_{E, E'}$ (see Lemma \ref{lem: construction_q_E,E}).

The construction of the principal polarization $\theta$ is more involved than the construction of the underlying abelian surface $A=E\times E'$. In this step, we numerically produce the principal polarization $\theta$, which can be represented as a $2\times 2$ Hermitian matrix (in terms of the map $\mathbf{D}$ defined in \sref{eq: isomorphism_NS(A)}). As a verification step, the constructed $(A,\theta)$ can be validated by computing its refined Humbert invariant; see Proposition \ref{prop: ref_hum_computation}.

 In Section \ref{s: Examples_Kaplansky}, rather than applying our algorithms to a generic geometric form, we choose a special ternary form $f$ (such that $4f$ is a geometric form) from Kaplansky's list \cite{kaplansky}. This allows us to translate Kaplansky's conjecture regarding representations of the form $f$ into the geometric assertion that a CM point $(A,\theta)$, where $q_{(A,\theta)}\sim 4f$, lies on the intersection of infinitely many Humbert surfaces. Consequently, determining whether this CM point lies in the intersection provides a potential path to resolving the conjecture.

We also give several remarks in Section \ref{s: Examples_Kaplansky} regarding the relation of the geometric form $q_C$ and the automorphism group $\Aut(C)$ and the elliptic subcovers of $C$.

It is of significant interest to determine the equations of curves $C/\Q$ of genus $2$ corresponding to geometric forms $q_{C}$ (where such equations exist).  It might be possible to obtain such equations in \textit{all} cases by combining our algorithm with the (extended) method described in \S4 of \cite{gelin2019principally}. We intend to investigate this possibility in future work.

\section{Ternary and binary quadratic forms}
\label{s: ternary_binary_forms}

In this article, we mostly work on positive definite binary and ternary integral quadratic forms as was mentioned in the introduction. For this, we want to use  results of Brandt \cite{brandt1952mass}, \cite{brandt1951zahlentheorie}, Cox \cite{davidcox}, Dickson \cite{dicksonsbook}, Jones \cite{jones1950arithmetic}, Smith \cite{smith} and Watson \cite{watson1960integral}. It is necessary for our aim to identify the assigned characters of a ternary quadratic form and its reciprocal and to calculate their values, and so we briefly mention some notations and concepts in the sense of those authors.

As defined in Watson \cite{watson1960integral} (and in Brandt \cite{brandt1951zahlentheorie}), an \textit{integral quadratic form} with $n$-variables is a polynomial $f(x_1,\ldots, x_n) $ of the form 
 \begin{equation}
\label{quadraticform}
    f \;=\; f(x_1,\ldots, x_n) \;=\; \sum_{1\leq i\leq j\leq n}a_{ij}x_ix_j, \textrm{ where } a_{ij}\in \Z.
\end{equation} 
Note that an integral quadratic form $f$ as in \sref{quadraticform} may not be an integral form in the sense of Dickson \cite{dicksonsbook} and Smith \cite{smith} because they additionally assume that the \textit{non-diagonal coefficients} of $f$ are even, i.e., $2\mid a_{ij},$ for all $i<j.$

The \textit{content}, denoted by $\cont(f)$, of the form $f$ is the greatest common divisor of its coefficients $a_{ij}$. If $\cont(f)=1$, we say that it is a \textit{primitive} form; otherwise, it is called \textit{imprimitive}. For the older terminology in \cite{dicksonsbook}, \cite{smith}, $f$ is called \textit{properly primitive} if it is primitive (in the above sense) and if $2\mid a_{ij},$ for all $i<j.$ Also, $f$ is called 
\textit{improperly primitive} if $\frac{1}{2}f$ is primitive (in the above sense) and if there exists some pair $(i,j)$ with $i<j$ such that $\frac{1}{2}a_{ij}$ is odd.

As usual, we say that two integral quadratic forms $f_1$ and $f_2$ with $n$-variables on $\Z^n$  are $\GL_n(\Z)$-\textit{equivalent}, (simply say equivalent) if one can be transformed into the other by an integral change of basis, denoted by $f_1\sim f_2$. Recall also that we define the further equivalence definitions in a similar way: If $f_1$ and $f_2$ are forms over $\Z_p$ and they are $\GL_n(\Z_p)$-equivalent (as in the obvious sense), where $\Z_p$ is the ring of $p$-adic integers, then we say that they are $p$-\textit{adically equivalent}. This is denoted by $f_1\sim_p f_2.$  In addition, if $f_1$ and $f_2$ are forms over $\mathbb{R}$ and if they are $\GL_n(\mathbb{R})$-equivalent, then we use the symbol  $f_1\sim_{\infty}f_2.$

Let $f_1$ and $f_2$ be two integral quadratic forms. If $f_1\sim_pf_2$ for all primes $p$ including $p=\infty,$ then we say that $f_1$ and $f_2$ are 
\textit{genus-equivalent} (or \textit{semi-equivalent}), denoted by $f_1 \simeq f_2$; cf.\ \cite[p.~106]{jones1950arithmetic}. Note that this is equivalent to the definition which was given in Watson \cite[p.~72]{watson1960integral}; cf.\ Theorem 40 of \cite{jones1950arithmetic} (also see Theorem 50 of \cite{watson1960integral}). In this case, we say that they are in the same \textit{genus}. Let $\gen(f)$ denote the set of equivalence classes of integral quadratic forms $f'$ which are genus-equivalent to a given integral quadratic form $f.$

Note that if $f_1$ and $f_2$ are positive definite properly primitive integral ternary forms, then they are genus-equivalent (in the above sense) if and only if they are genus-equivalent in Smith's sense \cite{smith} by Theorem 40 of \cite{jones1950arithmetic} and \S12 of \cite{smith}. Therefore, we study certain techniques for deciding when two positive definite integral ternary forms $f_1$ and $f_2$ are genus-equivalent by using the method of Smith \cite{smith} in order to prove our results and write our algorithms. This method was improved in \cite{brandt1952mass}, \cite{brandt1951zahlentheorie} by shortening the case distinctions. To avoid many case distinctions, we make use of this improvement. For this aim, we introduce some terminology for quadratic forms here.

Throughout the article, we usually talk about the integral quadratic forms, and we often drop the word integral when it is clear in the context.

By $\disc(f)$ and $\det(f),$ we mean the \textit{discriminant} and the \textit{determinant} of a quadratic form $f$ respectively. In particular, let $f_1=[a,b,c,r,s,t]$ denote the ternary quadratic form $ax^2+by^2+cz^2+ryz+sxz+txy$ and let $f_2=[a,b,c]$ denote the binary quadratic form $ax^2+bxy + cy^2$.
Then we have that \begin{equation}
    \label{eq: disc_formula_ternary_form}
    \disc(f_1) \;=\; -4abc \;-\; rst \;+\; ar^2 \;+\; bs^2 \;+\; ct^2 \quad \text{and} \quad \disc(f_2) \;=\; b^2 \;-4ac.
\end{equation}

If $f$ is properly or improperly primitive, then it has a \textit{reciprocal} quadratic form $F_f$ in Smith's sense as defined on \cite[p.~455]{smith} (or on \cite[p.~7]{dicksonsbook}). If $f$ is a primitive form, then it has \textit{reciprocal} as defined by Brandt \cite[p.~336]{brandt1951zahlentheorie}; this is denoted by $F_f^B.$ We will discuss the relations of these reciprocals when $f$ is an improperly ternary form; see Proposition \ref{F=F^B} below. 

To formulate these concepts explicitly, let $f=[a,b,c,r,s,t]$  be a positive definite, primitive integral ternary form. Let 
$$
M(f)\;=\;\begin{pmatrix}
  2a & t & s  \\
  t & 2b & r\\
   s & r & 2c
\end{pmatrix}
$$ 
denote the \textit{coefficient matrix} of $f$. Then the \textit{adjoint} $\adj(f)$ form of $f$ is defined by the formula as in \cite[p.~25]{watson1960integral}: 
 \begin{equation}
 \label{eq: adjoint_formula}
       M(\adj(f)) \;=\; -2\adj(M(f)),
   \end{equation} 
   where $\adj(M(f))$ is the usual adjoint of the matrix $M(f)$, i.e., the transpose of the cofactor matrix of $M(f)$.  The form $f$ form has a \textit{reciprocal} quadratic form defined in terms of the adjoint form $\adj(f)$ of $f$ as follows:
\begin{equation}
\label{eq: reciprocal_definition}
    \adj(f) \;=\; I_1(f) F_f^B,
\end{equation}
where    $F_f^B$  is a positive definite primitive form and   $I_1(f) \in \Z$.   We also define another quantity as
 $$
 I_2(f) \;:=\; I_1(F_f^B).
 $$
The quantites $I_1$ and $I_2$ are called  the \textit{basic (genus) invariants} of $f$ as in Brandt \cite[p.~316]{brandt1952mass}, \cite[p.~336]{brandt1951zahlentheorie}. 
 Smith also defined the \textit{basic (genus) invariants} $\Omega_f$ and $\Delta_f$ of $f$ as in \cite[p.~455]{smith} (and also see \cite[p.~7]{dicksonsbook}). Instead of defining the $\Omega$ and $\Delta$ , we discuss the relations of these quantities with the $I_1$ and $I_2$; see Proposition \ref{F=F^B} below.

By the identity on p.~316 of \cite{brandt1952mass}, we have for a primitive positive definite ternary form $f$ the useful formula that
 \begin{equation}
    \label{discf}
    \disc(f) \;=\; \frac{I_1(f)^2I_2(f)}{16}.
\end{equation}

We first derive from \cite{brandt1952mass} their relations here, and later,  we calculate them for the forms we studied on; cf.\ Proposition \ref{imprimitiveform} and \sref{eq: I_1(f_q)} below. 

\begin{prop} 
\label{F=F^B}
Let $f$ be a positive definite improperly primitive ternary form. Then $F_f$ is properly primitive, and $F_f=F_{f/2}^B.$ Moreover, we have that
\begin{align}
\label{I1odd}
     I_1(f/2)& \   \text{ is negative odd,  and } \quad 16\mid I_2(f/2), \\
         \label{I1Omega}
I_1(f/2)& \;=\; -\Omega_f \quad \text{ and } \quad I_2(f/2) \;=\; -8\Delta_f.
\end{align}
\end{prop}
\begin{proof}
Since $f$ is an improperly primitive ternary form, the primitive form $f/2$ satisfies the conditions of Case I of \cite[p.~316]{brandt1952mass}, and so $F_{f/2}^B=F_{f}$ is properly primitive and $I_1(f/2)$ is negative odd and $16\mid I_2(f/2)$ from the relations in \cite{brandt1952mass},  loc. cit. Moreover, we get from the relations of Type I of \cite{brandt1952mass}, loc. cit., that $I_1(f/2)=-\Omega_f \text{ and } I_2(f/2)=-8\Delta_f.$ Hence, all the assertions have been verified.
\end{proof}

Let us put $ \chi_{\ell}(x) = \left(\frac{x}{\ell}\right)$, for $x$  prime to  $\ell$,  where $\ell$  is an odd prime (where $(\frac{.}{.})$ is the Legendre-Jacobi symbol). If $x$ is odd, let us define $\chi_{-4}(x) = \left(\frac{-4}{x}\right) = (-1)^{(x-1)/2}$ and $\chi_8(x)=\left(\frac{8}{x}\right)=(-1)^{(x^2-1)/8}$. Given an integer $\delta,$ let 
$$
X(\delta) \;=\; \{\chi_{\ell}:\ell\textrm{ odd prime with }\ell\mid\delta\} \;\cup\; \{ \chi_{-4},\chi_{8}, \chi_{-4}\chi_{8}\}
$$
be the indicated set of  characters. If $f$ is a primitive ternary form or  $f$ is an improperly primitive ternary form and if $\chi_{n}\in X(\delta),$ for some $\delta,$ then $\chi_n$ is called an \textit{assigned character} of $f$  if the following relation holds:  
\begin{equation}
\label{assignedcharactersdefinition}
    \chi_{n}(r_1) \;=\; \chi_{n}(r_2), \text{ for any } r_1,r_2\text{ represented by }f\text{ with }\gcd(r_i,n)=1.
\end{equation} 
This common value is denoted by $\chi_{n}(f).$ 

Let $X(f)$ denote the set of the assigned characters of the primitive ternary form $f$  with the basic invariants $I_1$ and $I_2$, and let $F:=F_f^B$. If we let $X^*(\delta)=X(\delta)\setminus\{\chi_8,\chi_{-4}\chi_{8}\}$, then we see by Brandt \cite{brandt1951zahlentheorie}, \S19, that
\begin{align}
\label{eq1: set_assigned_char_brandtI1}
    &X(f) \;=\; X(I_1) \   \text{ if }\  32 \mid I_1, \text{ and }  X(f) \;=\; X^*(I_1) \ \text{ if } \ I_1 \equiv 16\Mod{32},\\
    \label{eq2: set_assigned_char_brandtI2}
     &X(F) \;=\; X(I_2)\  \text{ if } \  32\mid I_2, \text{ and }  X(F) \;=\; X^*(I_2)\  \text{ if } \  I_2 \equiv 16\Mod{32}.
\end{align}


Recall from \cite[p. 55]{davidcox} that we determine the set $X(q)$ of the assigned characters of a primitive binary quadratic form $q$ from the primes dividing its discriminant.

 If $q$ is a positive definite integral quadratic form, then let 
 \begin{equation}
 \label{eq: R(q)_notation}
     R(q) \;=\; \{q(x)>0:x\in \Z^n\}
 \end{equation}
 denote the set of positive numbers represented by $q$. In particular, if $x\in R(q)$ and $x$ is primitive, i.e., the gcd of the entries in the vector $x$ is $1$, then  we say that  $q$ \textit{represents primitively} $x$, and we write $q\rightarrow x$.

\section{The refined Humbert invariant}

\label{refinedhumbert}
In this section, we recall the concept of the refined Humbert invariant which was introduced by Kani \cite{kani1994elliptic}. We also discuss many useful 
 properties of the refined Humbert invariant as derived in \cite{kani2014jacobians}, \cite{MJ}, \cite{harun_1} and \cite{refhum}. This invariant is a quadratic form that is intrinsically attached to a principally polarized abelian surface.

Let $A/K$ be an abelian surface over an algebraically closed field $K$ of $\charec K = 0$. If $\lambda: A\stackrel\sim\rightarrow\hat{A}$ is a principal polarization of $A$, then $\lambda=\phi_{\theta}$ for a (unique) $\theta$ in the Neron-Severi group $\NS(A)=\Div(A)/\equiv,$ where $\equiv$ denotes numerical equivalence; (see \cite[p.~60]{mumford1970abelian} for the discussion of $\phi$). 
Thus, we can see the set of the principal polarizations of $A$ as a subset of $\NS(A)$. 

Let $\mathcal{P}(A) \subset \NS(A)$ denote the set of principal polarizations on $A$, and let $\theta\in \mathcal{P}(A)$. The quadratic form
$\tilde q_{(A,\theta)}$ of a principally polarized abelian surface $(A,\theta)$ on $\NS(A)$ is defined by the formula
\begin{equation}
\label{eq:qth} 
\tilde q_{(A,\theta)}(D) \;=\; (D.\theta)^2 - 2(D.D),\quad\mbox{for } D\in \NS(A),   
\end{equation}
where $(.)$ denotes the intersection number of divisors. By \cite[p.~200]{kani1994elliptic}, the form $\tilde q_{(A,\theta)}$ induces a positive definite quadratic form $q_{(A,\theta)}$ on the quotient module $\NS(A,\theta) = \NS(A)/\Z\theta$.

The quadratic module  $(\NS(A,\theta), q_{(A,\theta)})$ is called the \emph{refined Humbert invariant} of $(A,\theta)$; as defined in \cite{MJ}, \cite{kani2014jacobians}.  Thus the refined Humbert invariant $q_{(A,\theta)}$ gives rise to an equivalence class of integral quadratic forms in $\rho-1$ variables since $\NS(A,\theta)\simeq\mathbb{Z}^{\rho-1}$, where $\rho$ is the Picard number of $A$. 

 When  $A/K$ is a \textit{CM (abelian) product surface}, i.e., $A\simeq E_1\times E_2$ for some CM elliptic curves $E_1/K, E_2/K$, where $E_1$ and $E_2$ are isogenous, we have that $q:=q_{(A,\theta)}$ is a ternary quadratic form because the Picard number $\rho(A)=4$  in this case. We have also the converse of this, let us state it together with other useful results.

\begin{prop}
\label{imprimitiveform}
 (i) If $q_{(A,\theta)}$ is a ternary form, then $A\simeq E_1\times E_2$ is a CM product surface. Moreover, we have that 
 \begin{equation}
 \label{determinantoftherefined}
 \disc(q_{(A,\theta)}) \;=\; 16\disc(q_{E_1,E_2}),
 \end{equation}

 (ii) If $f:=q_{(A,\theta)}$ is a primitive ternary form, then $\Omega_f$ is even.

 (iii) If $f:=q_{(A,\theta)}$ is an imprimitive form, then $\Omega_{f/2}=\cont(q_{E_1,E_2})$
 \end{prop}
 \begin{proof}
     The assertions (i) and (iii) follow from Proposition 3.2 and Equation 4.2 of \cite{harun_1} respectively, and the assertion (ii) follows from \cite[p.316]{brandt1952mass}. Note that (\ref{determinantoftherefined}) originally comes from \cite{ESCI}.
 \end{proof}

 Recall from the introduction, we call a ternary refined Humbert invariant as a geometric form. As was mentioned in the introduction, we have a complete classification of the geometric forms. To state this, let us consider positive definite ternary quadratic forms $f(x,y,z)$ satisfying the following  conditions: If $f$ is primitive, then
\begin{align}
      \label{eq: class_ref_condition_1}
    &f(x,y,z)  \;\equiv\; 0,1 \Mod{4}, \ \text{ for all } \ x,y,z\in\Z,\\
    \label{eq: class_ref_condition_2}
    &f(x_0,y_0,z_0) \;=\; n^2 \ \textrm{ for some } \ x_0,y_0,z_0,n\in\Z \textrm{ with } \gcd(n,\disc(f))\;=\;1. 
\end{align}
If $f$ is imprimitive, then
\begin{align}
\label{classificationconditions1}
  \frac{1}{2}f& \text{ is an improperly primitive form, } \\
 f(x_0,y_0,z_0)& \;=\; (2n)^2 \  \textrm{ for some } \ x_0,y_0,z_0,n\in\Z \textrm{ with } \gcd(n,\disc(f))\;=\;1. \label{classificationconditions2}
\end{align}  We have the following classification result from \cite{harun_1, refhum} (cf.\ Theorem 4.1.1 of \cite{kir2024curvesPHD}):
\begin{thm}[Theorem 1 of \cite{refhum} and Theorem 1.2 of \cite{harun_1}]
\label{thm: classification_thm}
       Let $f$ be a positive definite ternary quadratic form. Then we have that if $f$ is primitive, then
       $$
      f \  \text{ is a geometric form } \ \Leftrightarrow \ f \ \text{ satisfies } \ \sref{eq: class_ref_condition_1} \ \text{ and } \  \sref{eq: class_ref_condition_2},
       $$
if $f$ is imprimitive, then
       $$
      f \  \text{ is a geometric form } \ \Leftrightarrow \ f \ \text{ satisfies } \ \sref{classificationconditions1} \ \text{ and } \  \sref{classificationconditions2}.
       $$
\end{thm}

\begin{thm}[Theorem 2 of \cite{refhum} (see Theorem 4.1.4 of \cite{kir2024curvesPHD})]
\label{thm: Theorem_2_ref_hum}
    Let $E_1 / K$ and $E_2 / K$ be two isogenous CM elliptic curves. If $f\in \gen(x^2+4{q_{E_1, E_2}})$, then there exists a principal polarization $\theta$ on $A=E_1 \times E_2$ such that $q_{(A, \theta)}\sim f$.
\end{thm}

Recall from \cite{milne1986jacobian} that if $C$ is a curve of genus $2$, then its Jacobian $J_C$ is an abelian surface
and there is a divisor $\theta_C$ on $J_C$, called the \textit{theta-divisor} such that $\theta_C$ is a principal polarization in $\NS(J_C)$. We then write $q_C := q_{(J_C,\theta_C)}$ for its associated refined Humbert invariant.  

One key property of the refined Humbert invariant $q_{(A,\theta)}$ is that it can be used to determine whether $(A,\theta)$ is a Jacobian. By Proposition 6 of \cite{MJ}, we have that 
\begin{equation} 
\label{thetaisirreducible}
(A,\theta) \simeq (J_C,\theta_C), \mbox{ for some curve }C\; \Leftrightarrow\; 
q_{(A,\theta)}(D)\neq 1, \ \forall D\in \NS(A,\theta).
 \end{equation}

If $A=E_1\times E_2$ is a product surface, where $E_1$ and $E_2$ are two elliptic curves, then we have an (explicit) isomorphism by Proposition 23 of \cite{MJ}:
\begin{equation}
\label{eq: isomorphism_NS(A)}
    \mathbf{D} :\Z \oplus\Z\oplus\Hom(E_1,E_2)\DistTo \NS(A).
\end{equation}

Let us present the computation of the refined Humbert invariant for a given degree form and a principal polarization here. One can use this result to verify the output of the main algorithms in this paper.

\begin{prop}
\label{prop: ref_hum_computation}
Let $A = E \times E'$ be a CM product surface over $K$. Let $\theta = \textbf{D}(n, m, kh) \in \mathcal{P}(A)$, where $h$ is a cyclic isogeny in $\Hom(E, E')$. Then we have a basis $\{h, h'\}$ of $\Hom(E, E')$, and let $q_{E, E'}(xh + yh') = [a, b, c]$. We also have
\[
q_{(A, \theta)} = \left[ n^2, \ m^2(mn + 3)a, \ b^2k^2 + 4c, \ -2bm(mn + 1), \ -2bkn, \ 2ka(mn + 2) \right].
\]
\end{prop}

\begin{proof}
Since $h$ is cyclic, there is $h' \in \Hom(E, E')$ such that $\Hom(E, E')$ has a basis $\{h, h'\}$. Therefore, we see from (\ref{eq: isomorphism_NS(A)}) that
$$
\beta:=\{D_1 = \textbf{D}(1, 0, 0), D_2 = \textbf{D}(0, 1, 0), D_3 = \textbf{D}(0, 0, h), D_4 = \textbf{D}(0, 0, h')\}
$$
gives a basis of $\NS(A)$. It is easy to construct a basis of $\NS(A)$ containing $\theta$ by the linear algebra. (Note that $\theta$ is primitive since $\theta$ is a principal polarization). 

By adapting the proof of Proposition 29 of \cite{MJ}, we see that $\theta, D_2, D_4$ and $D' := \textbf{D}(-k\deg(h), 0, -mh)$ forms a basis of $\NS(A)$. Alternatively, if we write the matrix $M_{\beta}$ of these elements with respect to the basis $\beta$, then we see that $\det(M_{\beta})=k^2\deg(h)-nm=1$ since $\theta$ is a principal polarization,  and so $D_2, D_4, D'$ gives a basis of $\NS(A) / \mathbb{Z}\theta$.

By applying the intersection number formula of any two divisors in (22) of \cite{MJ}, we can compute the refined Humbert invariant $q_{(A, \theta)}$, and the assertion follows.
\end{proof}

For an abelian surface $A$, let  $\mathcal{P}(A)^{\odd}:=\mathcal{P}(A)\setminus\mathcal{P}(A)^{\ev}$, where 
$$
\mathcal{P}(A)^{\ev} \;=\; \{\theta\in\mathcal{P}(A):   (D.\theta) \equiv 0 \Mod{2},\  \forall D\in \NS(A)\}.
$$
Let $A=E_1\times E_2$ be a product surface. While $\mathcal{P}(A)^{\odd}$ is always nonempty, $\mathcal{P}(A)^{\ev}$ may be an empty set for certain cases (cf.\ Proposition 3.6 of \cite{harun_1}). 
 
 The following proposition distinguishes several important results of the refined Humbert invariant $q_{(A,\theta)}$ depending on $\theta\in \mathcal{P}(A)^{\ev}$ or $\mathcal{P}(A)^{\odd}$.

\begin{prop}
\label{eventhetaclassification}
Let $A=E_1\times E_2$ be a CM product surface. Then we have that

\item[(i)] If $\theta\in \mathcal{P}(A)$, and $p$ is an odd prime or $p = \infty$, then $q_{(A,\theta)} \sim_p f_q$, where $f_q:=x^2+4q_{E_1,E_2}$. In particular, if $\theta\in \mathcal{P}(A)^{\text{odd}}$, then $q_{(A,\theta)} \in \gen(f_q)$ and $\cont(q_{(A,\theta)})=1$.

\item[(ii)] $\theta\in \mathcal{P}(A)^{\ev}$ if and only if $q_{(A,\theta)}$ is imprimitive. In this case, $\cont(q_{(A,\theta)})=4$.

\end{prop}

\begin{proof}
    The assertion (i) follows from Corollary 19 and Theorem 20 of \cite{kani2014jacobians}. The assertion (ii) follows from Corollary 3.5 and Proposition 4.1 of \cite{harun_1}. 
\end{proof}

\section{Construction of the abelian surface}

\label{S: section_constructionA}

In this section, for a given geometric form $f$, the aim is to construct a CM product abelian surface $A=E\times E'$ such that $f\sim q_{(A,\theta)}$, for some $\theta\in\mathcal{P}(A)$. To this end, for a given positive definite binary quadratic form $q$ we  explicitly construct CM elliptic curves $E/\C$ and $E'/\C$ such that $q\sim q_{E,E'}$ by following \cite{kani2011products}. We give an analytic construction of $E$ and $E'$, but one presents a known method to describe the elliptic curves algebraically by computing the \textit{j}-invariant of the associated lattices to the elliptic curves. We then see that this construction suffices to describe the abelian surface $A$ corresponding to the geometric form.

More precisely, assume that $f$ is a geometric form.  There is an algorithm giving a positive definite binary quadratic form $q$ such that $f$ lies in the genus of a ternary form associated with the form $q$; cf.\ Algorithms \ref{alg: return_q_primitive} and \ref{alg: return_q_imprimitive_g} below. If we construct CM elliptic curves $E/\mathbb{C}$ and $E'/\mathbb{C}$ such that $q\sim q_{E,E'}$, then we can conclude that there is a principal polarization $\theta$ on $A = E\times E'$ such that $f \sim q_{(A,\theta)}$ by Proposition \ref{p:construction_A} below. Recall that we can view a quadratic order $\mathcal{O}_{\Delta}$ of discriminant $\Delta<0$ as a lattice in $\C$ (or in $\Q(\sqrt{\Delta})$). The following lemma addresses the construction of CM elliptic curves as was discussed above. It essentially follows from several results of \cite{kani2011products}. 

\begin{lem}
\label{lem: construction_q_E,E}
    Let $q$ be a positive definite binary quadratic form with $\disc(q)=\Delta$ and $\cont(q)=\kappa$. Let us put $q':=\frac{1}{\kappa}q=[a, b, c]$, and let $\Delta':=\disc(q')=\Delta/\kappa^2$. If $L$ is the lattice of $\Q(\sqrt{\Delta'})$     with  basis  $\{a, (b+\sqrt{\Delta'})/2\}$, then the elliptic curve $E'=\mathbb{C}/L$ is isogenous to $E=\mathbb{C}/\mathcal{O}_{\Delta}$, and  we have that $q_{E,E'}\sim q$.
\end{lem}

\begin{proof}
  Let  $F = \Q(\sqrt{\Delta'})$ be an imaginary quadratic number field with discriminant $\Delta_F$.  Let $q_L$ be the binary form associated to $L$ as defined in \S3.3 of \cite{kani2011products}, i.e.,  
$$
q_L(x,y) \;=\; \frac{N(x\alpha + y\beta)}{N(L)},
$$
where $\{\alpha,\beta\}$ is a fixed basis of $L$ and $N(.)$ is the norm. If $L$ is the lattice with a basis $\{a, \frac{b+\sqrt{\Delta'}}{2}\}$ of $F$, then $N(L) = a$ (see \cite[p.~140]{davidcox}), and  it is easy to see that $q_L = [a, b, c] \sim q'$ (see Theorem 7.7 of \cite{davidcox}). 
    
 By (73) of \cite{kani2011products}, we see that $\disc(q_L)$ is equal to the discriminant $\Delta(R(L))$ of the order $R(L):=\{\alpha\in F : \alpha L\subset L\}$, and so $\Delta(R(L)) = \disc(q') =  \Delta'$. Thus, $R(L)$ is the (unique) order of $F$ of discriminant $\Delta'$.

Put $E'=\C/L$ and $E=\C/\mathcal{O}_{\Delta}$. Then by (68) of \cite{kani2011products} we have that $\End(E')\simeq R(L)$, and $\End(E)\simeq R(\mathcal{O}_{\Delta})=\mathcal{O}_{\Delta}$, so $E\sim E'$ by Proposition 36 of \cite{kani2011products}.

By Corollary 34 of \cite{kani2011products} we have that $I_E(E')\simeq \mathcal{O}_{\Delta}L^{-1}$, where $I_E(E'):=\Hom(E',E)\alpha$, with $\alpha: E\rightarrow E'$ is any isogeny. Since $\mathcal{O}_{\Delta}\subseteq R(L)$, we have that $I_E(E)I_E(E')^{-1}\simeq \mathcal{O}_{\Delta}(L^{-1})^{-1}=L$, and so by (74) of \cite{kani2011products} we obtain that $q_{E,E'}\sim \kappa 'q_L$, with $\kappa'=[R(L):\mathcal{O}_{\Delta}]=\kappa$. Thus, $q_{E,E'}\sim q$.
\end{proof}


Suppose that $f$ is a geometric form. Since we have by Proposition \ref{eventhetaclassification}(i) and (ii) that $\cont(f) = 1$ or $\cont(f) = 4$, the construction of a degree map $q_{E,E'}$ associated to $f$ is discussed in two cases based on the content of $f$.

\subsection{Case I: Primitive case}
Assume that $f$ is a geometric primitive form. Thus it follows by Proposition \ref{eventhetaclassification}(i)  that $f \in \gen(f_q)$, for some binary form $q$. If $q\sim q_{E, E'}$, for some elliptic curves $E$ and $E'$, then we see by Theorem \ref{thm: Theorem_2_ref_hum} that $f\sim q_{(E\times E',\theta)}$, for some principal polarization $\theta$ on $E\times E'$. Thus, we aim to determine this $q$, which will be denoted by $q_f$ later.

\begin{prop}
\label{p: primitive_alg}
    Suppose $f$ is a primitive geometric form. If there is a primitive positive definite binary form $q'$ with $\disc(q') = \frac{16\disc(f)}{I_1(f)^2}$ such that $\chi(F^B_f) = \chi(q')$, for all $\chi \in X(q')$, then  we have that $f\in  \gen(f_q)$, where $q= \frac{I_1(f)}{-16}q'$.
\end{prop}

\begin{proof}
    We first prove that $f$ and $f_q$ have the same basic invariants $I_1$ and $I_2$. As in \cite[p.~167]{kani2014jacobians},  when one calculates the adjoint form of $f_q$, it is not hard to see that
    \begin{equation}
    \label{eq: I_1(f_q)}
        I_1(f_q) = -16\cont(q) =: -16\kappa, \ \text{ and }\  F^B_{f_q} \sim F,\ \text{ where }\  F(0,y,z) = \frac{1}{\kappa}q(y,z).
    \end{equation}
 Since $q'$ is primitive, we have that $\kappa = \frac{I_1(f)}{-16}$, and so $I_1(f) = I_1(f_q)$. Moreover, it is easy to see from \sref{eq: disc_formula_ternary_form} that $\disc(f_q) = 16 \disc(q) =: 16\Delta$. Also, observe that $16\Delta = 16\disc(\kappa q') = 16\kappa^2\disc(q') = 16\kappa^2 \frac{16\disc(f)}{I_1(f)^2} = \disc(f)$, and so $\disc(f_q) = \disc(f)$. Therefore,  it follows by \sref{discf} that $\disc(q')=I_2(f) = I_2(f_q)$. 
 
 We know  by hypothesis that $f$ and $f_q$ are (properly) primitive. 
 Then, observe by what was explained in \S\ref{s: ternary_binary_forms} that $f\in \gen(f_q)$ once we have shown that
    \begin{equation}
    \label{eq: genus_equivalence_primitive}
        \chi(f) \;=\; \chi(f_q), \ \forall \chi\in X(f_q) \ \text{ and } \ \chi(F^B_f) \;=\; \chi(F^B_{f_q}), \ \forall \chi\in X(F_{f_q}^B). 
    \end{equation}

  Since $f$ is a geometric primitive form,  it  represents an integer $n^2$, for some $n$ with $\gcd(n,\disc(f))=1$ by \sref{eq: class_ref_condition_2}, and so it  follows that $\chi(f) = \chi(n^2)= 1$, for every $\chi \in X(f)$. This also holds for $f_q$ by the same result (or by the fact that $f_q(1,0,0) = 1$). Thus, the first equation of \sref{eq: genus_equivalence_primitive} holds.

  Since $q'$ is primitive, there is an integer $n$ prime to $2\disc(q')$ such that $F \rightarrow q' \rightarrow n$ by \sref{eq: I_1(f_q)} (and \cite[p.55]{davidcox}). Moreover, since $\disc(q')=I_2(f)$, we have that $X(F) = X(q')$ by \sref{eq2: set_assigned_char_brandtI2}, and so it follows that $\chi(F) = \chi(n) = \chi(q')$, for all $\chi \in X(F)$. But, we have by the hypothesis that $\chi(F^B_f) = \chi(q')$, for all $\chi \in X(q')$, which implies that $\chi(F^B_{f_q}) = \chi(F) = \chi(q') = \chi(F^B_f)$, for all $\chi \in X(F^B_{f_q})$, and so the second equation of \sref{eq: genus_equivalence_primitive} follows. Thus, the assertion is verified.
\end{proof}

By using Proposition \ref{p: primitive_alg}, we can write the following algorithm.

\begin{alg}
\label{alg: return_q_primitive}
\hfill \\ \textbf{Input:} A geometric primitive ternary form $f$.

\noindent\textbf{Output:} A binary form $q$ such that $f\in \gen(f_q)$.

1) Compute $\kappa = |I_1(f)|/16$ and $\Delta = \disc(f)/16$.

2) List all primitive positive definite (reduced) binary forms $q'$ of $\disc(q') = \Delta/\kappa^2$.

3) Compute the reciprocal form $F^B_f$, and compute $\chi(F^B_f)$, for all $\chi \in X(q')$. 

4) For the forms $q'$ found in Step 2, compute $\chi(q')$, for all $\chi \in X(q')$ 
until obtaining that $\chi(F^B_f) = \chi(q')$, for all $\chi \in X(q')$.

5) Return $q:=\kappa q'$.
\end{alg}


\begin{remark}
\label{rem: determine_assigned_characters}
i) In Step 2, we list primitive positive definite binary forms. Fortunately, this can be done efficiently; see Theorem 5.11.2 of \cite{vollmer2007binary}. 

ii)  Steps 3 and 4 require determining the values of the assigned characters of a primitive binary form $q$ and a primitive ternary form $f$. 
    
    For this, we first determine the set of assigned characters. Recall that $X(q)$ is determined by calculating the prime decomposition of $\disc(q)$ by \cite[p.~55]{davidcox}, and also that $X(f)$ is determined by the prime decomposition of $I_1(f)$ by \sref{eq1: set_assigned_char_brandtI1}.

    Then whenever we have $X(q)$ and $X(f)$, the computations of $\chi(q)$ and $\chi(f)$ can be done by employing some odd numbers $n_1 \in R(q)$ and $n_2 \in R(f)$ such that $\gcd(n_1, \disc(q)) = 1$ and $\gcd(n_2, I_1(f)) = 1$. One can produce such numbers by using well-known results related to the quadratic forms. Indeed, we  present an algorithm (for other purposes as well; see Algorithm \ref{alg: find_repr_assigned_value} and Remark \ref{rem: binary_rep_assigned_value}  below), and this algorithm can be used for finding $n_1$ and $n_2$.  
\end{remark}

\subsection{Case II: Imprimitive case}
  Assume that $f$ is a geometric imprimitive form. In Case I, (i.e., when $f$ is primitive) the useful fact was that $f\simeq f_q$, for a binary form $q$ (i.e., $f\in \gen(f_q)$). In Case II, we search a similar result. 
To this end, let us recall some notations and results from \cite{harun_1}.

Suppose that $f$ is a geometric imprimitive form with the basic genus invariants $I_1(f/4)$ and $I_2(f/4)$. Let\footnote{Note that this means that $I_1 = \Omega_{f/2}$ and $I_2=\Delta_{f / 2}$ by Proposition \ref{F=F^B}, and so we can use the relevant result from \cite{harun_1} accordingly.} $I_1:=|I_1(f/4)|$ and $I_2:=-I_2(f/4)/8$, and so $I_1$ and $I_2$ are positive odd and positive even integers, respectively, by Proposition \ref{F=F^B}.

By Proposition 4.7 of \cite{harun_1}, there exists a positive integer $a$ such that:
\begin{equation}
\label{eq: a_conditions}
    aI_1 \equiv 3 \Mod{4} \quad \text{ and }  -2I_2 \text{ is quadratic residue in modulo } a.
\end{equation}
Thus, there exists a positive integer $b$ such that $-2I_2 \equiv b^2 \pmod{a}$. We may assume that $b$ is even by replacing $b$ with $a-b$, if necessary. Hence, it follows that $-2I_2 \equiv b^2 \pmod{4a}$ (since $I_2$ is even), and so we write:
\begin{equation}
\label{eq: q_Omega_Delta}
    q_{I_1, I_2}(x,y) \;:=\; I_1 a x^2 \;+\; I_1 b xy \;+\; I_1 \frac{b^2 + 2I_2}{4a} y^2
\end{equation}
which is a positive definite binary quadratic form of discriminant $-2I_2I_1^2$ with content $I_1$ (cf.\ Equation (4.6) of \cite{harun_1}). By Lemma \ref{lem: construction_q_E,E}, there exist CM elliptic curves $E_1$ and $E_2$ such that $q_{E_1,E_2}\sim   q_{I_1, I_2}$. Moreover, if we let $ A := E_1 \times E_2$, then  it follows from Proposition 4.10 of \cite{harun_1} that
\begin{equation}
\label{eq: f_genus_eq_q_Omega_Delta}
    f \;\simeq\; f^{\theta}_{q_{I_1,I_2}} \;:=\; q_{(A, \theta)}, \ \text{ for some } \ \theta \in \mathcal{P}(A).
\end{equation}
 Therefore, we write an analog Algorithm  \ref{alg: return_q_primitive} for the imprimitive case as follows:



\begin{alg}
\label{alg: return_q_imprimitive_g}
\hfill \\ \textbf{Input:} A geometric imprimitive form $f$.

\noindent\textbf{Output:} A binary form $q_{I_1,I_2}$ 
such that $f\simeq f^{\theta}_{q_{I_1,I_2}}$, for some $(A,\theta)$ as in \sref{eq: f_genus_eq_q_Omega_Delta}.

1) Compute the basic invariants $I_1(f/4)$ and $I_2(f/4)$, and put $I_1:=|I_1(f/4)|$ and $I_2:=-I_2(f/4)/8$.

2) Find  $a>0$ such that $\left(\frac{-2I_2}{p}\right) = 1$, for all prime $p\mid a$, and $\left(\frac{-1}{a}\right) =-\left(\frac{-1}{I_1}\right)$.

3) Find an even integer $b>0$ such that $-2I_2 \equiv b^2 \pmod{4a}$.

4) Return:
    $q_{I_1, I_2}(x,y) = I_1 a x^2 + I_1 b xy + I_1 \frac{b^2 + 2I_2}{4a} y^2$.
\end{alg}
\begin{proof}[Proof of Algorithm]
This algorithm is written just by following the above discussion after we note that $\left(\frac{-1}{aI_1}\right) =-1$ if and only if $aI_1 \equiv 3 \Mod{4}$.
 Also, note that there exists a prime number $p=a$ such that there exists an even integer $b$ such that  $-2I_2 \equiv b^2 \pmod{4a}$ by Proposition 3.5.2 of \cite{kir2024curvesPHD}, and so we just could search a prime number $a$ to return the output.
\end{proof}

\noindent\textbf{Notation:} For an input $f$, let $q_f$ denote an output of Algorithm \ref{alg: return_q_primitive} if $f$ is a geometric primitive form, and of Algorithm \ref{alg: return_q_imprimitive_g} if $f$ is a geometric imprimitive form.

It should be noted that there is no necessity for $q_f$ to be unique for a given geometric form $f$. Indeed, it is also delicate to find all possible such $q_f$ because they present all possible decomposition of the surface $A$; cf.\ Remark \ref{rem: const_surfaces_all_A} below.

The following proposition shows that Algorithms \ref{alg: return_q_primitive} and \ref{alg: return_q_imprimitive_g} addresses the task of the construction of the abelian surface that is the aim of this section.

\begin{prop}
    \label{p:construction_A}
    Let $f$ be a geometric form. Let $q:=q_f$, and let $\kappa = \cont(q)$, and $\frac{1}{\kappa}q =: q'=[a,b,c]$, and let $\Delta = \disc(q)$ and $\Delta' = \disc(q')$. Put $A = \mathbb{C}/\mathcal{O}_{\Delta} \times \mathbb{C}/L$, where $L \subset \Q(\sqrt{\Delta'})$ is the lattice with basis  $\{a, \frac{b+\sqrt{\Delta'}}{2}\}$.
    Then there is a principal polarization $\theta$ on $A$ such that $f\sim q_{(A,\theta)}$.
    \end{prop}

\begin{proof} 
Since $q$ is a positive definite form, we have from Lemma \ref{lem: construction_q_E,E} that there exist two isogenous CM elliptic curves $E:=\mathbb{C}/\mathcal{O}_{\Delta}$ and $E':=\mathbb{C}/L$ such that $q\sim q_{E,E'}$. Let $A:=E\times E'$ be a CM product surface. If $f$ is primitive, then we have that $f \in \gen(f_q)$; cf.\ Algorithm \ref{alg: return_q_primitive}. Then, there exists a principal polarization $\theta$ on $A$ such that $q_{(A, \theta)} \sim f$ by Theorem \ref{thm: Theorem_2_ref_hum}.

In a similar way, if $f$ is imprimitive, then $f\simeq q_{(A,\theta')}$, for some $(A,\theta')$; cf.\ Algorithm \ref{alg: return_q_imprimitive_g}. Thus,  by Theorem 1.1 of \cite{harun_1}, there exists a principal polarization $\theta$ on $A$ such that $f \sim q_{(A,\theta)}$.  
Thus, the assertion follows.
\end{proof}

\begin{remark}
\label{rem: const_surfaces_all_A}
  For a given geometric form $f$, this proposition gives one abelian surface $A$ such that $f \sim q_{(A,\theta)}$, for some $\theta \in \mathcal{P}(A)$. It is actually possible to give the complete list of all isomorphism classes of $A$'s with this property, i.e., we can describe all elements of the set
    \begin{equation}
    \label{eq: definition_N(f)}
       \{ A/K:  q_{(A,\theta)} \sim f, \text{ for some } \theta \in \mathcal{P}(A) \}/\!\simeq
    \end{equation}
since there is a bijection between this set and the set $\gen(q_f)/\!\sim$ (cf.\ Proposition 4.3.4 of \cite{kir2024curvesPHD}).
\end{remark}

\section{Construction of the principal polarization}
\label{s: construction_pp}

For a given geometric form $f$ we explicitly constructed an abelian surface $A$ such that $f\sim q_{(A,\theta)}$, for some  $\theta \in \mathcal{P}(A)$ in the previous section. 
We now give a method for the construction of a principal polarization $\theta$. Although the construction has common ideas in both the primitive and imprimitive cases, we have to consider them separately in our approach because of the essential differences of the arithmetic properties of the forms. 

When $f$ is primitive,  since our method (of the constructing the principal polarization) essentially comes from the recipes of the proof of main result of \cite{refhum} (cf.\ Theorem 1 of \cite{refhum}), it is indeed a refinement of a sketchy algorithm outlined in this paper; see Algorithm 27 of \cite{refhum}. To state steps of this construction, we introduce the following set and a  binary form, which are important in the following:

Given an integer $d\geq1$, let
 \begin{equation}
 \label{eq: definition_P(d)}
      P(d) \;:=\; \{(n, m, k)\in\Z^3 : n, m>0  \textrm{ and } nm-k^2d=1 \},
 \end{equation}
 and for $s=(n,m,k)\in P(d),$ let\footnote{It is worth to mention that the formula in \sref{q_s} can be derived from the formula in Proposition \ref{prop: ref_hum_computation} by putting $z=0$ and $d:=\deg(h)=a$, where $z$ is the third variable of the ternary form. The reason is that $q_s$ is the refined Humbert invariant of $(A,\theta)$ when $A=E\times E'$, where the degree map on $\Hom(E,E')$ is $ax^2$.}
 \begin{equation}
 \label{q_s}
    q_s(x,y) \;:\; =n^2x^2+2kd(nm+2)xy+m^2d(nm+3)y^2.
\end{equation} 

Following \cite{MJ}, a binary quadratic form $q$ is said to be of \textit{type d} if $q\sim q_s$, for some $s\in P(d)$. Let $q$ be a binary quadratic form of discriminant $-16d$. Then we have an important result by Theorem 13 of \cite{MJ} that
\begin{equation}
\label{eq: Theorem_13_MJ}
    q \  \text{ has type } d \ \Leftrightarrow \   \frac{q}{c}\ \text{ lies in the principal genus, where } c:=\cont(q)=1 \text{ or } 4.
\end{equation}

Recall the notation $R(q)$ from \sref{eq: R(q)_notation}.
When $f$ is a primitive geometric form,  here are the steps of this construction briefly: 
\begin{alg}[Primitive case] Assume that $f$ is a primitive geometric form.
\label{m: method_1_primitive}
\begin{enumerate}
    \item Apply Algorithm \ref{alg: return_q_primitive} to $f$ to obtain $q_f=:q$.
     \item  Find a triple $(x,y,z) \in \Z^3$ such that $F^B_f(x,y,z)=p$ is an odd prime and $p \nmid \disc(q):=D$; see Proposition \ref{prop: BW_existence_prime}.
  \item  Find a binary form $\Tilde{q} \in \gen(\frac{1}{\cont(q)}q)$ such that $p\in R(\Tilde{q})$; see Lemma \ref{lem: binary_form_represent_p_in_proof}.
\item  Apply Algorithm \ref{alg: produce_phi_p} to $F^B_{f}$ and the triple $(x,y,z)$ to return a properly primitive  binary form $\phi_p$ of  $\disc(\phi_p) = -16\cont(q)p$ represented by $f$. 
\item Apply Algorithm \ref{alg: produce_s=(n,m,k)} to $\phi_p$ to return a triple $s = (n, m, k)$ such that $q_s \sim \phi_p$.
\item Construction step: Write $E:=\C/\mathcal{O}_{D}$ and $E':=\C/L$, where $L$ is the lattice with the basis $\{p, \frac{-b+\sqrt{\disc(\tilde{q})}}{2}\}$, where $b$ is the $xy$ coefficient of $\tilde{q}(x,y)$. Then put $A:=E\times E'$, and let $\theta:=\textbf{D}(n,m,kh)$, where $h$ is a cyclic isogeny of degree $p\cont(q)$.
\item  Verification Step: Use Proposition \ref{prop: ref_hum_computation} for $(A,\theta)$ to find $q_{(A,\theta)}$, and check whether $q_{(A,\theta)} \sim f$ or not.
\end{enumerate}
\end{alg}

By following the steps of this algorithm, we can construct a principally polarized abelian surface $(A,\theta)$ such that $f\sim q_{(A,\theta)}$; see Theorem \ref{thm: construction_A_theta} below.

 In order to deal with Step 2 in Algorithm \ref{m: method_1_primitive}, we use Proposition \ref{prop: BW_existence_prime}, and so we assume GRH  only in this step.  

When $f$ is an imprimitive,  here are the  steps of the construction briefly:
\begin{alg}[Imprimitive case] Assume $f$ is an imprimitive geometric form.
\label{m: method_2_imprimitive}
\begin{enumerate}
    \item Apply Algorithm \ref{alg: return_q_imprimitive_g} to $f$, return $q:=q_f$.
     \item Find a triple $(x,y,z) \in \Z^3$ such that $F^B_{f/4}(x,y,z)=p$ is an odd prime and $p \nmid \disc(q):=D$; see Proposition \ref{prop: BW_existence_prime}. 
  \item  Find a binary form $\Tilde{q} \in \gen(\frac{1}{\cont(q)}q)$ such that $p\in R(\Tilde{q})$; see Lemma \ref{lem: binary_form_represent_p_in_proof}.
\item Apply Algorithm \ref{alg: produce_phi_p} to $F^B_{f/4}$ and the triple $(x,y,z)$ to return a  binary form $\phi_p$ of content $2$ and $\disc(\phi_p) = -4\cont(q)p$ represented by $f/2$. 
\item Apply Algorithm \ref{alg: produce_s=(n,m,k)} to $2\phi_p$ to return a triple $s = (n, m, k)$ with $q_s \sim 2\phi_p$. 
\item Construction step: Write $E:=\C/\mathcal{O}_{D}$ and $E':=\C/L$, where $L$ is the lattice with the basis $\{p, \frac{-b+\sqrt{\disc(\tilde{q})}}{2}\}$, where $b$ is the $xy$ coefficient of $\tilde{q}(x,y)$. Then put $A:=E\times E'$, and let $\theta:=\textbf{D}(n,m,kh)$, where $h$ is a cyclic isogeny of degree $p\cont(q)$.
\item Verification Step: Use Proposition \ref{prop: ref_hum_computation} for $(A,\theta)$ to find $q_{(A,\theta)}$, and check whether $q_{(A,\theta)} \sim f$ or not.
\end{enumerate}
\end{alg}



For a given primitive binary form $f$ and a given integer $n$, we  need to find a value $m$ represented by $f$ such that $\gcd(m, n) = 1$. In the proof of Theorem 67 of \cite{jones1950arithmetic}, there is a way to find such a value. Let us write it as an algorithm here:

\begin{alg}
\label{alg: find_repr_assigned_value}
\hfill \\ \textbf{Input:} A primitive binary form $f = [a, b, c]$ and an integer $n \geq 1$.

\noindent\textbf{Output:} $(x, y)$ and $f(x, y)=N$ such that $\gcd(x, y) = \gcd(N,n)=1$.

1) Apply the factorization algorithm to determine the prime decomposition of $n = p_1^{\alpha_1} \cdots p_k^{\alpha_k}$. For $i$ from $1$ to $k$:

\quad (i) If $p_i \nmid a$, set $(x_i, y_i) = (1, 0)$,

\quad (ii) if $p_i \nmid c$, set $(x_i, y_i) = (0, 1)$,

\quad (iii) if $p_i \mid \gcd(a,c)$,

\quad \quad \quad if $p_i \nmid b$, set $(x_i, y_i) = (1, 1)$.

 2) By the Chinese Remainder Theorem, there is a unique tuple $(x', y')$ such that $0 \leq x', y' \leq \prod_{i=1}^k p_i,$ and such that $ x'\equiv x_i \Mod{p_i}, \ y'\equiv y_i \Mod{p_i}$ for all $i$.
 
 3) Set $g = \gcd(x', y')$, and return $(x,y)=(x'/g, y'/g)$ and $N=f(x,y)$.
\end{alg}

\begin{remark}
\label{rem: binary_rep_assigned_value}
    It is clear that  Algorithm \ref{alg: find_repr_assigned_value} can be generalized to primitive ternary forms as well just by following similar steps as was clearly described in the proof of    Theorem 6 of \cite{dicksonsbook}.
\end{remark}

To address Step 4 of Algorithm \ref{m: method_1_primitive}, we obtain an algorithm which produces a binary form $\phi_N$ of $\disc(\phi_N) = -4\Omega_f N$ represented by $f$ provided that $N\in R(F_f)$.   By the recipe of the proof of Theorem 38 of \cite{dicksonsbook}, we have the following algorithm for finding such a binary form. 
Before stating it, note that by a \textit{primitive transformation matrix} $T$ of a ternary form $f$ and a binary form $\phi$, we mean that $T$  is a $3\times 2$ matrix whose entries are integers and the gcd of its $2\times 2$ minors is $1$ such that $fT=\phi$, i.e., $(fT)(x)=f(Tx)=\phi(x)$, for any $x\in \Z^2$.

\begin{alg}
\label{alg: produce_phi_p}
\hfill \\ \textbf{Input:} A properly or improperly primitive ternary form $f$ and $(x_0,y_0,z_0)\in \Z^3$ such that $\gcd(x_0,y_0,z_0)=1$.

 \noindent\textbf{Output:} A primitive transformation matrix $T$ of $f$ and $\phi_N:=fT$, where $\phi_N$ is a binary quadratic form of $\disc(\phi_N) = -4\Omega_fF_f(x_0,y_0,z_0)$.

1) If $z_0 \neq 0$, then set $g := \gcd( z_0, x_0 + y_0)$. Then, put $a_1 = a_2 = z_0/g$ and $a_3 = -(x_0 + y_0) / g $.

2) Apply the extended Euclidean algorithm to $a_1l+a_2m+a_3n = 1$ to obtain a solution $(l, m, n)$. (Such a solution exists because $\gcd(a_1, a_2, a_3) =1$). Set $b_1 = a_1 - mz_0 + ny_0$, $b_2 = a_2 - nx_0 + lz_0$, and $b_3 = a_3-ly_0 + mx_0$.


3) Return the transformation matrix $T:= [[a_1, a_2, a_3], [b_1, b_2, b_3] ]^t$ and $\phi_N := fT$.

4) If $z_0 = 0$, then at least one of $x_0$ or $y_0$ is not zero. We know that $f(x,y,z)\sim f(x,z,y) \sim f(z,y,x)$. So, return Step 1 to apply $f(x,z,y)$ and $(x_0,z_0,y_0)$  if $y_0 \neq 0$, and to apply  $f(z,y,x)$ and $(z_0,y_0,x_0)$  if $x_0 \neq 0$.
\end{alg}

Finding a solution $(l, m, n)$ of $a_1l+a_2m+a_3n = 1$ is easy in Step $2$ of this algorithm by the extended Euclidean algorithm. Indeed, we can find $(\alpha', \beta')$ such that $\gcd(a_1, a_2) = \alpha'a_1 + \beta' a_2$ by the Euclidean algorithm. Since $\gcd(a_1, a_2, a_3) = 1$, we can similarly find $(\alpha, \beta)$ such that $\gcd(a_1, a_2)\alpha + a_3\beta = 1$, and thus,
$$
1 \;=\; \gcd(a_1, a_2)\alpha \;+\; a_3\beta \;=\; (\alpha'a_1 \;+\; \beta' a_2)\alpha \;+\; a_3\beta \;=\; \alpha'\alpha a_1 \;+\; \beta'\alpha a_2 \;+\; a_3\beta,
$$
and so,  $(l, m, n) = (\alpha\alpha', \alpha\beta', \beta)$ is a solution as required.

It is worth noting that if $z_0=0$ and $y_0 \neq 0$, then Step 4 of this algorithm indicates that we should restart the process by interchanging the roles of $y$ and $z$ in all steps. Consequently, we must consider $f(x,z,y)$  and $F_f(x,z,y)$ instead of $f(x,y,z)$  and $F_f(x,y,z)$ in this new process. Similarly, if $z_0 = 0$ and $x_0 \neq 0$, we should take $f(z,y,x)$ and $F_f(z,y,x)$.

If $\phi$ is a binary form of \textit{type} $d$, there is a triple  $s = (n, m, k) \in P(d)$ (see \sref{eq: definition_P(d)})  such that $q_s \sim \phi$ (see \sref{q_s}) by definition. To complete Step 5 of Algorithms \ref{m: method_1_primitive} and \ref{m: method_2_imprimitive}, we need to find the triple $(n,m,k)$ explicitly. To this end,  since we use the recipe of the proof of Theorem 13 of \cite{MJ},  we first need an algorithm which produces a square $N^2$ represented by a given binary quadratic form of type $d$ with $\gcd(N^2, d) = 1$. Recall from \sref{eq: Theorem_13_MJ} that either $\phi$ is in the principal genus or $\phi = 4\phi'$, for some form $\phi'$ which is in the principal genus. This guarantees the existence of $N^2\in R(\phi)$.

Finding such a square number is intrinsically linked to Gauss’s duplication theorem; see Article 286 of \cite{Gauss}. This result asserts that if $q$ is a primitive form belonging to the principal genus, then there is a form $q_1$ such that $q \sim q_1\circ q_1$, for some $q_1$ with $\disc(q_1) = \disc(q)$. By Theorems 68 and 72 of \cite{jones1950arithmetic}, this implies that  $q$ represents a square relatively prime to its discriminant. 
An efficient algorithm to construct $q_1$ for a given $q$  is detailed in \cite[\S 3.4, Thm.~6.7]{Stevenhagen_2_class_groups}. The authors of \cite{Stevenhagen_2_class_groups} describe an explicit algorithm (and claim that it is implemented in MAGMA). We omit the details here, as the description is lengthy. For our purposes, computational complexity is not a primary concern. It therefore suffices to determine an upper bound for the represented perfect square, allowing us to complete the algorithm solely via the representation problem.



\begin{prop}
\label{p: existence_square_integer}
Suppose that $q$ is a primitive positive definite binary form of discriminant $-\Delta$ lying in the principal genus. Then there is an integer $N$ with $N \le \Delta^{3}$ such that $\gcd(N, \Delta) = 1$ and $q$ primitively represents $N^2$.
\end{prop}

\begin{proof}
By Gauss's duplication theorem, there is a primitive form $q_1$ such that $q_1 \circ q_1 \sim q$ with $\disc(q_1) = -\Delta$. We may suppose $q_1$ is reduced, and write $q_1 = [a, b, c]$. 

If we follow Algorithm \ref{alg: find_repr_assigned_value}, then we see that there are integers $x,y$ such that $0 \leq x, y \leq \prod_{p\mid \Delta}p$, where $p$ is prime and $q_1(x,y)=N_1$ is relatively prime to $\Delta$.

Recall that we have $|b| \le a \le c$ and $3ac \le \Delta$ since $q_1$ is reduced. Then we see
\begin{align*}
N_1\;=\;q_1(x,y)\;=\; ax^2 + bxy + cy^2 \le c\Delta^2 + c\Delta^2  + c\Delta^2= 3c\Delta^2 \le \Delta^3.
\end{align*}
Since $q_1$ represents $N_1$, it follows that $q \sim q_1 \circ q_1$ represents $\frac{N_1^2}{\gcd(a,b)^2}=:N^2$ by the duplication formula; e.g., see Corollary 4.13 of \cite{Duncanbinaryforms}. Thus, $q(x,y)=N^2$ such that $N$ is relatively prime to $\Delta$ with $N \le \Delta^{3}$. If we put $g:=\gcd(x,y)$, then the assertion follows since it is obvious that $q(x/g,y/g)$ is also square satisfying the desired properties. 
\end{proof}

While it is certainly possible to derive a sharper bound, the current result suffices for our purposes; thus, we avoid further analysis.

 The following algorithm addresses Step 5 of Algorithm \ref{m: method_1_primitive}.

\begin{alg}
\label{alg: produce_s=(n,m,k)}
\hfill \\ \textbf{Input:} A binary form $\phi$ of \textit{type} $d$. 

\noindent\textbf{Output:} A triple $s = (n, m, k) \in P(d)$ such that $q_s \sim \phi$.

1) Find an integer $N$ such that $\phi \rightarrow N^2$ and $\gcd(N,d)=1$.

2) Apply Lemma 2.3 of \cite{davidcox} to $\phi$ and $N^2$, then find  $\Tilde{\phi} = [N^2, 2b, c] \sim \phi$. 


3) If $N$ is odd, then solve $k \equiv (2d)^{-1}b \Mod{N^2}$ for (the smallest positive) $k$. 

\quad Return $s = (N, (k^2d+1)/N, k)$.

 \quad If $N$ is even, then solve $d^{*} \equiv d^{-1} \Mod{N^2}$ for (the smallest positive) $d^*$.  Put $\Bar{c} \equiv c \Mod{4}$, where $\Bar{c} \in \{0,1\}$, and put $k = d^*(b/2+\Bar{c} N/2)$.

  \quad Return $s = (N, (k^2d+1)/N, d^{*}(b + \Bar{c} N)/2)$.
\end{alg}

This algorithm follows from the proof of part (i) $\Rightarrow$ (iii) of  Theorem 13 of \cite{MJ}. In Step 1, we apply Proposition \ref{p: existence_square_integer} to guarantee the existence of such an integer satisfying an explicit bound. In Step 2, if $\phi(x,y)=N^2$, then find $z,w$ such that $xz-yw=1$, and so the transformation matrix $\begin{pmatrix}
x & w \\
y & z
\end{pmatrix}$ gives the desired form $\tilde{\phi}$.


The following lemma is a part of the main algorithm; cf.\ Step 3 of Algorithms \ref{m: method_1_primitive} and \ref{m: method_2_imprimitive}.

\begin{lem}
\label{lem: binary_form_represent_p_in_proof}
    Let $q$ be a positive definite binary quadratic form with $\disc(q)=d$. If $q$ represents an odd prime number $p\nmid d$, then there exists an integer $b$ such that $q$ is equivalent to $[p, b, \frac{b^2-d}{4 p}]$.
\end{lem}
\begin{proof}
    If $q$ represents a prime number, then it is clear that $q$ is a primitive form. Since $p \nmid d$ is represented by a primitive form $q$ of discriminant $d$, it follows by Lemma 2.5 of \cite{davidcox} that $d$ is a quadratic residue modulo $p$, which means that there is an integer $b$ such that $b^2-d\equiv 0 \Mod{p}$. By replacing $b$ with $p-b$, if necessary, we may assume that $b\equiv d\Mod{2}$. Recall that $d\equiv 0,1 \Mod{4}$ since $d$ is a discriminant of a binary form. Thus $b^2\equiv d \Mod{4}$, and so $b^2-d\equiv 0 \Mod{4p}$.

    We now consider the form $q':=[p, b, \frac{b^2-d}{4 p}]$. Since $\gcd(d,p)=1$, it follows that $\gcd(b,p)=1$, and so $q'$ is a primitive form, and its discriminant is $d$. Hence, since the forms $q'$ and $q$ have the same discriminant and represent the same prime number $p$, they are equivalent (cf.\ Exercise 2.27 of \cite{davidcox}), which proves the assertion.
    \end{proof}

It remains to address Step 2 of our main algorithm: finding a prime number represented by a ternary quadratic form. The following result resolves this step.

\begin{prop}[GRH][Proposition III.1 of \cite{BW_hard_problems}]
\label{prop: BW_existence_prime}
Assume that GRH holds. There exists a constant $c$ and
an algorithm $\mathscr{A}$ such that the following holds. Let $f$ be
a primitive, positive definite, integral quadratic form of
dimension $n$. For any $M > (2^{n^2} |\disc(f)|)^c$, the algorithm $\mathscr{A}(f, M)$ outputs a vector $x \in Z^n$ such that $f(x)$ is a prime number between $M$ and $M^2$. 

In particular, for a given integer $N$, the algorithm $\mathscr{A}(f, N_{\epsilon})$ outputs a vector $x\in \Z^n$ such that $f(x)$ is a prime number, $\gcd(f(x),N))=1$ and $f(x) \leq N_{\epsilon}^2$, where $N_{\epsilon}= \max((N+1), (2^{n^2} |\disc(f)|)^c)$.
\end{prop}

\begin{remark}
   In the proof of Proposition III.1 of \cite{BW_hard_problems}, the relevant algorithm is explained in detail. We, however, rely on an upper bound to locate a prime represented by a binary or ternary form; this allows us to proceed by simply calling the representation problem. While Theorems Theorems 4.4.3 and 4.4.22 of \cite{kir2024curvesPHD} give an explicit bound of a prime number with additional specific properties for binary forms, it is excessively large compared to Proposition \ref{prop: BW_existence_prime}. Furthermore, extending that result to ternary forms would require further adaptation.
\end{remark}

We are ready to prove the correctness of our main algorithms, which give the construction of $(A,\theta)$ such that $q_{(A,\theta)} \sim f$ for a given geometric form $f$.

\begin{thm}[GRH]
\label{thm: construction_A_theta}
Assume GRH holds. For a given geometric primitive form $f$, we construct $(A,\theta)/\C$ such that $f \sim q_{(A,\theta)}$ by following the steps of Algorithm \ref{m: method_1_primitive}.
\end{thm}

\begin{proof}
  By Step 1, we know that $f \in \gen(f_q)$, where $q:=q_f$ is an output of Algorithm \ref{alg: return_q_primitive}.     Let $D= \disc(q)$ and $\kappa = \cont(q)$. 
  
  By Proposition \ref{prop: BW_existence_prime}, we find a prime number $p$ relatively prime to $\disc(f) = 16D$ (cf.\ \sref{determinantoftherefined}) such that $F_f^B(x_0, y_0, z_0) = p$. By Algorithm \ref{alg: produce_phi_p}, we find a binary form $\phi_p$ of discriminant $-4\Omega_fp$ represented properly by $f$. Since $\gcd(p, 16D) = 1$, it follows that $\phi_p:=[u,2w,v]$ is a properly primitive or an improperly primitive binary form by Theorem 37 of \cite{dicksonsbook}. Indeed, the latter is not possible. If it were improperly primitive, then we would have that $\gcd(u,w,v)=1$ and $2\mid \gcd(u,v)$. Since $\disc(\phi_p)/4=w^2-uv=-\Omega_fp$ and since $\Omega_f$ is even by Proposition \ref{imprimitiveform}, it follows that $2\mid w$, which is a contradiction. Thus, $\phi_p$ is properly primitive.
    
    By Theorem 22 of \cite{refhum} (see also Theorem 4.1.6 of \cite{kir2024curvesPHD}), we see that $\phi_p$ lies in the principal genus, and so, it follows by Theorem 13 of \cite{MJ} that it is a primitive form of type $\kappa p$ as was mentioned above. Applying Algorithm \ref{alg: produce_s=(n,m,k)}, we obtain a triple $s= (n, m, k) \in P(\kappa p)$ such that $q_s \sim \phi_p$.

    Let us put $D' := D/\kappa^2$. By Corollary 21 of \cite{refhum} (see also Corollary 4.1.7 of \cite{kir2024curvesPHD}), there exists a form $q_1\in \gen(\frac{1}{\kappa}q)$ which represents $p$. By Lemma \ref{lem: binary_form_represent_p_in_proof}, there is an integer $b$ such that $\Tilde{q} : = [\kappa p, \kappa b, \kappa \frac{b^2-D'}{4p} ]\sim \kappa q_1$, and so $f\in \gen(f_q) = \gen(f_{\Tilde{q}})$. 

Let $E:=\mathbb{C}/\mathcal{O}_{D}$ and $E':=\mathbb{C}/L$ be two elliptic curves, where $L$ is the lattice with the basis  $\{p, \frac{-b+\sqrt{D'}}{2}\}$ of $\Q(\sqrt{D'})$. By Lemma \ref{lem: construction_q_E,E}, it follows that $q_{E,E'} \sim \Tilde{q}$. Since $\Tilde{q}(1,0) = \kappa p$, there is a primitive isogeny $h$ in $\Hom(E,E')$ such that $\deg(h) = \kappa p$.  To make this isogeny more explicit, note that there is a basis $\{ \alpha, \beta\}$ of $\Hom(E,E')$ such that $q_{E,E'}(1\alpha + 0\beta) = q_{E,E'}(\alpha) = \kappa p$, and so we take $h = \alpha$. 
    
    Put $A = E\times E'$ and $\theta = \textbf{D}(n,m,kh)$.  As in the proof of Theorem 1 of \cite{refhum}, it follows that $f\sim q_{(A,\theta)}$, and thus, the assertion follows. 
    \end{proof}

\begin{thm}[GRH]
\label{thm: construction_A_theta_Imprimitive}
Assume GRH holds.    For a given geometric imprimitive form $f$, we construct $(A,\theta)$ such that $f \sim q_{(A,\theta)}$  by following the steps of Algorithm \ref{m: method_2_imprimitive}.
\end{thm}

\begin{proof}
  By Step 1, we first apply Algorithm \ref{alg: return_q_imprimitive_g} to $f$ to obtain a binary form $q_f =: q$. Let $\cont(q) = \kappa$ and $\disc(q)= D$, and also put $D'=D/\kappa^2$. 
    
    Since $f/2$ is improperly primitive by \sref{classificationconditions1}, we have that $F:=F_{f/2}=F^B_{f/4}$ is properly primitive by Proposition \ref{F=F^B}. By Proposition \ref{prop: BW_existence_prime}, we have a prime number $p = F(x_0, y_0, z_0)$ such that $p\nmid \disc(f)$, and so $p\nmid \kappa D'$ (because $16D=\disc(f)$ by \sref{determinantoftherefined}). By Algorithm \ref{alg: produce_phi_p}, we obtain a binary form $\phi_p$ of $\disc(\phi_p) = -4\Omega_{f/2}p = -4\kappa p$ (by Proposition \ref{imprimitiveform}) such that $f/2 \rightarrow \phi_p$. Moreover, we have that $\cont(\phi_p)=2$ by Proposition 5.3 of \cite{harun_1}. By Corollary 5.5 of \cite{harun_1}, $2\phi_p \sim q_s$, for some $s\in P(\kappa p)^{\ev}$, i.e., $2\phi_p $ is a binary form of type $\kappa p$, and so, we find $s=(n,m,k)$ by Algorithm \ref{alg: produce_s=(n,m,k)}.

     By Proposition 5.2 of \cite{harun_1}, there is a binary form $\Tilde{q} \in \gen(q)$ such that $\kappa p \in R(\Tilde{q})$. Since $p\nmid D'$ is represented by a primitive binary form $\frac{1}{\kappa}\Tilde{q}=\Tilde{q}'$ of discriminant $D'$, there exists an integer $b$ such that $g':=[p,b,\frac{b^{2}- D'}{4 p} ]\sim \Tilde{q}'$ by Lemma \ref{lem: binary_form_represent_p_in_proof}.

Let $E:=\mathbb{C}/\mathcal{O}_{D}$ and $E':=\mathbb{C}/L$, where $L$ is the lattice with the basis  $\{p, \frac{-b+\sqrt{D'}}{2}\}$ of $\Q(\sqrt{D'})$. By Lemma \ref{lem: construction_q_E,E}, it follows that $q_{E,E'} \sim \Tilde{q} = \kappa \Tilde{q}' \sim \kappa g'=:g$. Since $g(1,0) = \kappa p$, there is a primitive isogeny $h$ in $\Hom(E,E')$ such that $\deg(h) = \kappa p$.  To make this isogeny more explicit, note that there is a basis $\{ \alpha, \beta\}$ of $\Hom(E,E')$ such that $q_{E,E'}(1\alpha + 0\beta) = q_{E,E'}(\alpha) = \kappa p$, and so we can take $h = \alpha$. 
    
    Put $A = E\times E'$ and $\theta = \textbf{D}(n,m,kh)$. By the proof of Theorem 1.1 of \cite{harun_1}, $f\sim q_{(A,\theta)}$, and so the assertion follows.
\end{proof}

Observe that Theorem \ref{thm: main} follows from Theorems \ref{thm: construction_A_theta_Imprimitive} and  \ref{thm: construction_A_theta}.

After we validated the main Algorithms \ref{m: method_1_primitive} and \ref{m: method_2_imprimitive}, we provide them in a simpler way as follows. The following algorithm is more convenient to implement in a computer software like SageMath \cite{SageMath}.

\begin{alg}[Simplified Main algorithm] Assume $f$ is a geometric form, let $\kappa=\cont(f)$, and let $f'=\frac{f}{\kappa}$.
\label{m: simplified_algorithm}
\begin{enumerate}
     \item Find a triple $(x,y,z) \in \Z^3$ such that $F^B_{f'}(x,y,z)=p$ is an odd prime and $p \nmid \disc(f')$; see Proposition \ref{prop: BW_existence_prime}. 
  \item  Find a binary form $\Tilde{q}$ with $D:=\disc(\Tilde{q})=\disc(f)/16$ such that $p\in R(\Tilde{q})$; see Lemma \ref{lem: binary_form_represent_p_in_proof}.
\item Apply Algorithm \ref{alg: produce_phi_p} to $F^B_{f'}$ and $(x,y,z)$ to return a  binary form $\phi_p$.
\item Apply Algorithm \ref{alg: produce_s=(n,m,k)} to $\phi_p$ or $2\phi_p$ if $\kappa=1$ or $\kappa=4$, respectively, to return a triple $s = (n, m, k)$ such that $q_s \sim \phi_p$ or $q_s\sim 2\phi_p$ respectively. 
\item Construction step: Write $E:=\C/\mathcal{O}_{D}$ and $E':=\C/L$, where $L$ is the lattice with the basis $\{p, \frac{-b+\sqrt{\disc(\tilde{q})}}{2}\}$, where $b$ is the $xy$ coefficient of $\tilde{q}(x,y)$. Then put $A:=E\times E'$, and let $\theta:=\textbf{D}(n,m,kh)$, where $h$ is a cyclic isogeny of degree $p\cont(q)$.
\item Verification Step: Use Proposition \ref{prop: ref_hum_computation} for $(A,\theta)$ to find $q_{(A,\theta)}$, and check whether $q_{(A,\theta)} \sim f$ or not.
\end{enumerate}
\end{alg}

Recall from Proposition \ref{eventhetaclassification} that $\cont(f)=1$ or $4$, for a geometric form $f$. Also, we have from \sref{determinantoftherefined} that $\disc(q_f)=\disc(f)/16$, and so the discriminant $D$ in the second step of Algorithm \ref{m: simplified_algorithm} is the same of the second steps of Algorithms \ref{m: method_1_primitive} and \ref{m: method_2_imprimitive}. The other steps of Algorithm \ref{m: simplified_algorithm}  are the same with the ones of \ref{m: method_1_primitive} and \ref{m: method_2_imprimitive}.

\section{Examples}
\label{s: Examples_Kaplansky}

\subsection*{Kaplansky's Conjecture and Humbert Surfaces}
Recall that Humbert (1899) introduced an analytic invariant, now called the \textit{Humbert invariant} to classify abelian surfaces $A/\C$, with $\End(A)\neq \Z$. This invariant is an integer. Then the \textit{Humbert surfaces} $H_N$ are defined via the Humbert invariant $N$;  see van der Geer \cite{van2012hilbert}, Ch. IX, for a contemporary treatment of Humbert surfaces.  The refined Humbert invariant is an algebraic generalization of the (usual) Humbert invariant, and it also provides an algebraic description of the Humbert surfaces due to Kani \cite{MJ}:

Given a quadratic form $q(x)=Nx^2$, for a positive integer $N$, the Humbert surface $H_N$ of invariant $N$ is defined by
\begin{equation}
\label{eq: Humbert_surface_def}
    H_N \;=\; \{\langle A,\theta\rangle  \in \mathcal{A}_2(K)\, :\,  q_{(A,\theta)} \rightarrow  N\},
\end{equation}
where  $\mathcal{A}_2(K)$ is the moduli space of the principally polarized abelian surfaces over an algebraically closed field $K$.

First, consider the intersection of infinitely many Humbert surfaces:
$$
\mathcal{H}_{\infty} \;:=\; \bigcap H_{ 4N}, \ \text{ where } N \text{ runs over all positive odd square-free integers. }
$$

Second, consider the positive definite ternary quadratic form $f' = [1, 3, 7, 0, 1, 1]$. This is a special form appearing in line 23 of Kaplansky's list\footnote{Note that we could choose any form in this list except the ones in lines 1-3.} \cite{kaplansky}. Kaplansky conjectured that this form represents all positive odd integers, and Rouse \cite{Jeremy_Rouse} showed that this conjecture holds assuming the GRH.

We now translate this conjecture into geometric language. First, observe that $f :=4f'= [4, 12, 28, 0, 4, 4]$ is a geometric form by Theorem \ref{thm: classification_thm} since $f/2$ is improperly primitive and $f$ represents $2^2$ obviously, and so we see from \sref{thetaisirreducible} that $f\sim q_C$, for some genus $2$ curve $C$.

If the conjecture holds, then it is clear that $4f'=f$ primitively represents $4N$, for any square-free odd integer $N$. Thus, we have by definition (see \sref{eq: Humbert_surface_def}) that $(J_C, \theta_C) \in H_{4N}$ for any odd square-free integer $N$, and so $(J_C, \theta_C) \in \mathcal{H}_{\infty}$.

Conversely, if $(J_C, \theta_C) \in \mathcal{H}_{\infty}$, then it follows by the definition that the form $q_C$ represents $4N$ for any square-free integer $N$. This clearly implies that $q_C$ represents $4N$ for any odd integer $N$ as well. Thus, $\frac{1}{4}q_C = [1, 3, 7, 0, 1, 1]=f'$ represents all odd integers. 

Therefore, we have that Kaplansky's conjecture holds if and only if $(J_C, \theta_C) \in \mathcal{H}_{\infty}$, (where $q_C\sim 4f')$. In the next example, we apply Algorithm \ref{m: method_2_imprimitive} to $4f'$ to find $(J_C, \theta_C)$, and so Kaplansky's conjecture will be equivalent to the fact that the explicit $(J_C, \theta_C)\in \mathcal{H}_{\infty}$.

\begin{ex}\footnote{Note that the correctness of our Example does not rely on the GRH.}
\label{ex: kaplansky_example}
Let $f = [4, 12, 28, 0, 4, 4]$. We find $(J_C,\theta_C)$ such that $q_C\sim f$ by applying Algorithm \ref{m: method_2_imprimitive} to $f$.

\textbf{Step 1.} We apply Algorithm \ref{alg: return_q_imprimitive_g} to $f$. This means that we first compute the basic invariants $I_1(f/4)$ and $I_2(f/4)$, and set:
$I_1:=|I_1(f/4)|=1$ and $I_2:=-I_2(f/4)/8=148$. Then we find that $a=3$, and that $b=2$. Thus Algorithm \ref{alg: return_q_imprimitive_g} returns a binary form $q:=q_{I_1,I_2}=[3, 2, 25]$ of discriminant $-2I_1^2I_2=-296$.

\textbf{Step 2.} We compute the adjoint form of $f/2$:
The adjoint form is given by:
\[
F_{f/2} \;=\; F_{f/4}^B \;=\; [84, 27, 11, 2, -12, -28].
\]
We see that $F_{f/2}(0,0,1)=11$, and $11\nmid -296=2^3.37$.

\textbf{Step 3.} By following the proof of Lemma \ref{lem: binary_form_represent_p_in_proof}, we find an even integer $b$ such that  $b^2 - \disc(q) \equiv 0 \Mod{11}$. Since $b^2 \equiv -296 \equiv 1 \Mod{11}$, the value $b=12$ holds, and so
this step returns the binary form $\tilde{q}:=[11, 12, 10]$.

\textbf{Step 4.} Since we have from Step 2 that $F_{f/2}(0,0,1) = 11$, we apply Algorithm \ref{alg: produce_phi_p} to $f/2$ and the vector $(0,0,1)$, and it returns the transformation matrix: $\begin{pmatrix} 1 & 1 & 0 \\ 0 & 1 & 0 \end{pmatrix}^t$
and the binary form $\phi = [10, 14, 6]$ of discriminant $-44$.

\textbf{Step 5.} We apply Algorithm \ref{alg: produce_s=(n,m,k)} to $2\phi = [20, 28, 12]$, which is a form of type $11$.
 We find that $2\phi(1, -1) = 4 =: N^2,$ where $N=2$.
Using the transformation matrix $\begin{pmatrix} 1 & 0 \\ -1 & 1 \end{pmatrix}$, the form $2\phi$ transforms to $[4, 4, 12]$.
Since $N=2$ is even, we calculate $s = (2, \frac{9 \cdot 11 + 1}{2}, 3) = (2, 50, 3)$.

\textbf{Step 6.} Let $E := \C / \mathcal{O}_{-296}$ and $E' := \C / L$, where $L$ is the lattice with basis $\{ 11, \frac{-12 + \sqrt{-296}}{2} \}$.
Put $A := E \times E'$, and let $\theta = \textbf{D}(2, 50, 3h)$, where $h$ is an isogeny in $\Hom(E, E')$ with $\deg(h) = 11$. We fix a basis for $\Hom(E, E')$ such that $\deg(\alpha x + \beta y)=\tilde{q}=[11,12,10]$, and set $h=\alpha$.

\textbf{Step 7.} If we use Proposition \ref{prop: ref_hum_computation} for  $\tilde{q}=[11,12,10]$ and $\theta = \textbf{D}(2, 50, 3h)$, then we see that $q_{(A,\theta)}= [4, 2832500, 1336, -121200, -144, 6732]$ whose Eisenstein reduced form is $f$, and so $q_{(A,\theta)}\sim f$, as expected.
\end{ex}

\begin{remark}
Let $(A,\theta)$ be as in Example \ref{ex: kaplansky_example}, and so $(A,\theta)=(J_C,\theta_C)$, for some curve $C$ of genus $2$, and $q_C \sim [4,12,28,0,2,2]$. It seems that $C$ is a “special” CM point in the sense that the discriminant of the elliptic curves components of $J_C$ are relatively small. Also, since $q_C$ primitively represents $4$ at two points, namely $q_C(\pm1,0,0)=4$, it follows from Proposition 35 of \cite{ESCII} that $C$ has one elliptic involution, i.e., $\Aut(C)$ is not trivial \footnote{Indeed, one can see that $\Aut(C)=C_2\times C_2$ by Theorem 29 of \cite{cas}, where $C_2$ denotes the cyclic group of order $2$.}, and hence, $C$ is also special in this sense. Therefore, it would be highly interesting to find a way to show that this special CM point is an element of $\mathcal{H}_{\infty}$ to conclude that Kaplansky's conjecture holds.
\end{remark}

\subsection*{Elliptic Subcovers and Automorphisms}

Let $C/K$ be a curve of genus 2 over $K$. Recall from \cite{kani1994elliptic} that an \textit{elliptic subcover} is a finite morphism $\phi: C\rightarrow E$ to an elliptic curve $E/K$ which does not factor over a non-trivial isogeny (i.e., minimal covering) of $E$. Its degree is the degree of the morphism.

To present the relation  of elliptic subcovers and the refined Humbert invariant,  let us recall from
Theorem 4.5 of Kani \cite{kani1994elliptic} that for a curve $C/K$ of genus 2 we have
\begin{equation}
    \label{eq: elliptic_subcover_fact}
    C/K \text{ has an elliptic subcover of degree } N \;\Leftrightarrow\; q_C \rightarrow N^2.
\end{equation}

Kani (see Theorem 29 of \cite{cas} and cf.\ Table 1 of \cite{Automorphism}) showed that the number of solutions of the equation $q_C=4$ determines the automorphism group $\Aut(C)$ of a genus 2 curve $C$. Moreover, we have a list of geometric forms $q_C$ encoding the information $\Aut(C)$; cf.\ the main results of \cite{Automorphism} and Tables 5 and 6. 

For example, there is a characterization of genus 2 curves whose automorphism group contains (up to isomorphism) the dihedral group $D_6$ of order $12$ as follows:
\begin{equation}
\label{eq: aut_group_list_D6}
   D_6 \leqslant \Aut(C) \  \Leftrightarrow \  q_C\sim [4, 4, 4] \text{ or } [4,4,c,4,4,4] \text{ or } [4, 4, c, 0, 0, -4], 
\end{equation}
where $\ c \equiv 0, 1 \Mod{4}$ and $c>1$ (see Theorem 1.2 of \cite{Automorphism}).

By combining these results regarding a genus 2 curve $C$, we can argue interesting properties about curves having an elliptic subcover with prescribed degree and specific automorphism groups in terms or geometric forms, and then apply our algorithm to find a representative of the curve. For example, the set of genus $2$ curves whose automorphism group is $D_6$ with having an elliptic subcover of degree $3$ are determined in Theorem 1 of \cite{Shaska_subcovers}. As in Example 6.3 of \cite{Automorphism}, we can characterize the same curves by the refined Humbert invariant as
$$
[4,4,5,4,4,4], [4,4,9,4,4,4], [4,4,5,0,0-4] \ \text{ and } \ [4,4,9,0,0-4],
$$
then we apply Algorithm \ref{m: simplified_algorithm} to find the representatives of the curves. Since the discriminants of these forms are very small, illustrating of applying our algorithms is not quite interesting. So, let us just write the outputs of Algorithm \ref{m: simplified_algorithm} by SageMath here:

\begin{table}[h]
\centering
\small 
\setlength{\tabcolsep}{4pt} 
\caption{Algorithm \ref{m: simplified_algorithm} Results. }
\label{tab:alg_outputs}
\begin{tabular}{l c c c c}
\hline
 $f=[a,b,c,r,s,t]$ & $D$ & $\tilde{q}$ & $p$ & $s=(n,m,k)$ \\
\hline
$[4, 4, 5, 4, 4, 4]$ & $-11$ & $[3, 1, 1]$ & $3$ & $(2, 122, 9)$ \\
$[4, 4, 9, 4, 4, 4]$ & $-23$ & $[3, 1, 2]$ & $3$ & $(2, 122, 9)$ \\
$[4, 4, 5, 0, 0, -4]$ & $-15$ & $[23, 13, 2]$ & $23$ & $(2, 104, 3)$ \\
$[4, 4, 9, 0, 0, -4]$ & $-27$ & $[39, 21, 3]$ & $13$ & $(2, 7, 1)$ \\
\hline
\end{tabular}
\end{table}

{\footnotesize
\noindent \textbf{Acknowledgments.} I thank to my PhD supervisor, Prof.\ Ernst Kani, for his guidance and support since the core of this work originates from my doctoral studies under his supervision.
}

\bibliographystyle{plainurl}
\bibliography{pp}

@article{harun_1,
author = {K\i{}r, H.},
title = {The classification of the refined Humbert invariant for curves of genus 2},
journal = {International Journal of Number Theory},
volume = {21},
number = {06},
pages = {1247-1279},
year = {2025},
doi = {10.1142/S1793042125500654},

URL = { 
    
        https://doi.org/10.1142/S1793042125500654
    
    

},


}

@INPROCEEDINGS{BW_hard_problems,
  author={Wesolowski, B.},
  booktitle={2021 IEEE 62nd Annual Symposium on Foundations of Computer Science (FOCS)}, 
  title={The supersingular isogeny path and endomorphism ring problems are equivalent}, 
  year={(2022)},
  volume={},
  number={},
  pages={1100-1111},
  doi={10.1109/FOCS52979.2021.00109}}

@book {Duncanbinaryforms,
    AUTHOR = {Buell, D. A.},
     TITLE = {Binary Quadratic Forms. Classical Theory and Modern Computations},
      
 PUBLISHER = {Springer-Verlag, New York},
      YEAR = {(1989)},
    PAGES = {x+247},
      OPTISBN = {0-387-97037-1},
   MRCLASS = {11E16 (11-02)},
  MRNUMBER = {1012948},
MRREVIEWER = {A. G. Earnest},
       OPTDOI = {10.1007/978-1-4612-4542-1},
       URL = {https://doi.org/10.1007/978-1-4612-4542-1},
}

@article {brandt1951zahlentheorie,
    AUTHOR = {Brandt, H.},
     TITLE = {Zur {Z}ahlentheorie der tern\"{a}ren quadratischen {F}ormen},
   JOURNAL = {Math. Ann.},
  FJOURNAL = {Mathematische Annalen},
    VOLUME = {124},
      YEAR = {(1952)},
     PAGES = {334--342},
      OPTISSN = {0025-5831,1432-1807},
   MRCLASS = {10.0X},
  MRNUMBER = {51269},
MRREVIEWER = {R.\ Hull},
       OPTDOI = {10.1007/BF01343574},
       URL = {https://doi.org/10.1007/BF01343574},
}

@article {brandt1952mass,
    AUTHOR = {Brandt, H.},
     TITLE = {\"{U}ber das {M}ass positiver tern\"{a}rer quadratischer
              {F}ormen},
   JOURNAL = {Math. Nachr.},
  FJOURNAL = {Mathematische Nachrichten},
    VOLUME = {6},
      YEAR = {(1952)},
     PAGES = {315--318},
      OPTISSN = {0025-584X,1522-2616},
   MRCLASS = {10.0X},
  MRNUMBER = {51268},
MRREVIEWER = {R.\ Hull},
       OPTDOI = {10.1002/mana.19520060507},
       URL = {https://doi.org/10.1002/mana.19520060507},
}

@article {kani2011products,
    AUTHOR = {Kani, E.},
     TITLE = {Products of {CM} elliptic curves},
   JOURNAL = {Collect. Math.},
  FJOURNAL = {Collectanea Mathematica},
    VOLUME = {62},
      YEAR = {(2011)},
    NUMBER = {3},
     PAGES = {297--339},
      OPTISSN = {0010-0757},
   MRCLASS = {11G10 (11G15 14H52 14K02 14K22 14L15)},
  MRNUMBER = {2825715},
MRREVIEWER = {Joseph H. Silverman},
       OPTDOI = {10.1007/s13348-010-0029-1},
        URL = {https://doi.org/10.1007/s13348-010-0029-1},
}

@book {vollmer2007binary,
    AUTHOR = {Buchmann, J. and Vollmer, U.},
     TITLE = {Binary quadratic forms},
    SERIES = {Algorithms and Computation in Mathematics},
    VOLUME = {20},
      NOTE = {An algorithmic approach},
 PUBLISHER = {Springer, Berlin},
      YEAR = {(2007)},
     PAGES = {xiv+318},
      ISBN = {978-3-540-46367-2},
   MRCLASS = {11E16 (11-01 11R11 11Y40)},
  MRNUMBER = {2300780},
MRREVIEWER = {Duncan\ A.\ Buell},
}

@book {Gauss,
    AUTHOR = {Gauss, C. F.},
     TITLE = {Disquisitiones arithmeticae},
      NOTE = {Translated and with a preface by Arthur A. Clarke},
 PUBLISHER = {Springer-Verlag, New York},
      YEAR = {1986},
     PAGES = {xx+472},
      ISBN = {0-387-96254-9},
   MRCLASS = {01A75 (01A55)},
  MRNUMBER = {837656},
}

@article {Stevenhagen_2_class_groups,
    AUTHOR = {Bosma, W. and Stevenhagen, P.},
     TITLE = {On the computation of quadratic {$2$}-class groups},
   JOURNAL = {J. Th\'{e}or. Nombres Bordeaux},
  FJOURNAL = {Journal de Th\'{e}orie des Nombres de Bordeaux},
    VOLUME = {8},
      YEAR = {(1996)},
    NUMBER = {2},
     PAGES = {283--313},
      ISSN = {1246-7405,2118-8572},
   MRCLASS = {11R29 (11R11 11Y40)},
  MRNUMBER = {1438471},
MRREVIEWER = {F.\ Diaz y Diaz},
       URL = {http://jtnb.cedram.org/item?id=JTNB_1996__8_2_283_0},
}

@article {kani2014jacobians,
    AUTHOR = {Kani, E.},
     TITLE = {Jacobians isomorphic to a product of two elliptic curves and ternary quadratic forms},
   JOURNAL = {J. Number Theory},
  FJOURNAL = {Journal of Number Theory},
    VOLUME = {139},
      YEAR = {(2014)},
     PAGES = {138--174},
      OPTISSN = {0022-314X,1096-1658},
   MRCLASS = {14H40 (11G10 11G15 11G18 14H30 14H52)},
  MRNUMBER = {3173190},
MRREVIEWER = {Edward\ F.\ Schaefer},
       OPTDOI = {10.1016/j.jnt.2013.12.006},
       URL = {https://doi.org/10.1016/j.jnt.2013.12.006},
}

@article {ESCI,
    AUTHOR = {Kani, E.},
     TITLE = {Elliptic subcovers of a curve of genus 2. {I}. {T}he isogeny defect},
   JOURNAL = {Ann. Math. Qu\'{e}.},
  FJOURNAL = {Annales Math\'{e}matiques du Qu\'{e}bec},
    VOLUME = {43},
      YEAR = {(2019)},
    NUMBER = {2},
     PAGES = {281--303},
      OPTISSN = {2195-4755,2195-4763},
   MRCLASS = {14H30 (14H05 14H25 14H40)},
  MRNUMBER = {3996071},
MRREVIEWER = {Giancarlo\ Urz\'{u}a},
       OPTDOI = {10.1007/s40316-018-0105-6},
       URL = {https://doi.org/10.1007/s40316-018-0105-6},
}

@article {ESCII,
    AUTHOR = {Kani, E.},
     TITLE = {Elliptic subcovers of a curve of genus 2 {II}. {T}he refined
              {H}umbert invariant},
   JOURNAL = {J. Number Theory},
  FJOURNAL = {Journal of Number Theory},
    VOLUME = {193},
      YEAR = {(2018)},
     PAGES = {302--335},
      OPTISSN = {0022-314X,1096-1658},
   MRCLASS = {14H30 (11G30 14H40)},
  MRNUMBER = {3846811},
MRREVIEWER = {Sajad\ Salami},
       OPTDOI = {10.1016/j.jnt.2018.05.011},
       URL = {https://doi.org/10.1016/j.jnt.2018.05.011},
}

@incollection {SubcoversofCurves,
    AUTHOR = {Kani, E.},
     TITLE = {Subcovers of curves and moduli spaces},
 BOOKTITLE = {Geometry at the frontier},
    SERIES = {Contemp. Math.},
    VOLUME = {766},
     PAGES = {229--250},
 PUBLISHER = {Amer. Math. Soc., RI},
      YEAR = {(2021)},
      OPTISBN = {978-1-4704-5327-5},
   MRCLASS = {14G35},
  MRNUMBER = {4248056},
MRREVIEWER = {Nicolae\ Manolache},
       OPTDOI = {10.1090/conm/766/15384},
       URL = {https://doi.org/10.1090/conm/766/15384},
}

@article{PP,
  title={Principal Polarizations on {$E\times E'$} },
  author={Kani, E.},
   year={Preprint, 17 pages},
  
}

@article{refhum,
  title={The refined {Humbert} invariant for
abelian product surfaces with complex
multiplication },
  author={Kani, E.},
   year={Preprint, 23 pages},
  
}

@article{cas,
  title={ Curves of genus 2 on abelian surfaces },
  author={Kani, E.},
   year={Preprint, 37 pages},
  
}

@article {MJ,
    AUTHOR = {Kani, E.},
     TITLE = {The moduli spaces of {J}acobians isomorphic to a product of
              two elliptic curves},
   JOURNAL = {Collect. Math.},
  FJOURNAL = {Collectanea Mathematica},
    VOLUME = {67},
      YEAR = {(2016)},
     PAGES = {21--54},
      OPTISSN = {0010-0757},
   MRCLASS = {14H10 (14H40)},
  MRNUMBER = {3439838},
MRREVIEWER = {Francisco J. Plaza Mart\'{\i}n},
       OPTDOI = {10.1007/s13348-015-0148-9},
       URL = {https://doi.org/10.1007/s13348-015-0148-9},
}

@article{Automorphism,
  title={The Refined {H}umbert Invariant for a Given Automorphism Group of a Genus 2 Curve},
  author={Kır, H.},
   year={to appear, 28 pages},
URL = {https://arxiv.org/pdf/2310.19076},
  
}

@incollection {Shaska_subcovers,
    AUTHOR = {Shaska, T.},
     TITLE = {Genus 2 curves with {$(3,3)$}-split {J}acobian and large
              automorphism group},
 BOOKTITLE = {Algorithmic number theory ({S}ydney, 2002)},
    SERIES = {Lecture Notes in Comput. Sci.},
    VOLUME = {2369},
     PAGES = {205--218},
 PUBLISHER = {Springer, Berlin},
      YEAR = {(2002)},
      ISBN = {3-540-43863-7},
   MRCLASS = {14H37 (14H45)},
  MRNUMBER = {2041085},
MRREVIEWER = {Sadok\ Kallel},
       DOI = {10.1007/3-540-45455-1\_17},
       URL = {https://doi.org/10.1007/3-540-45455-1_17},
}

@book {van2012hilbert,
    AUTHOR = {van der Geer, G.},
     TITLE = {Hilbert modular surfaces},
    SERIES = {\emph{Ergebnisse der Mathematik und ihrer Grenzgebiete (3)}},
    VOLUME = {16},
 PUBLISHER = {Springer-Verlag, Berlin},
      YEAR = {(1988)},
     PAGES = {x+291},
      OPTISBN = {3-540-17601-2},
   MRCLASS = {11F41 (11G10 11G15 14J20)},
  MRNUMBER = {930101},
MRREVIEWER = {O.\ V.\ Shvartsman},
       OPTDOI = {10.1007/978-3-642-61553-5},
       URL = {https://doi.org/10.1007/978-3-642-61553-5},
}

@inproceedings {gelin2019principally,
    AUTHOR = {G\'{e}lin, A. and Howe, E. and Ritzenthaler,
              C.},
     TITLE = {Principally polarized squares of elliptic curves with field of
              moduli equal to {$\mathbb{Q}$}},
 BOOKTITLE = {Proc. 13th {A}lgorithmic {N}umber
              {T}heory {S}ymposium},
    SERIES = {Open Book Ser. 2},
    VOLUME = {},
     PAGES = {257--274},
 PUBLISHER = {},
      YEAR = {(2019)},
      OPTISBN = {978-1-935107-03-3; 978-1-935107-02-6},
   MRCLASS = {11G15 (14H25 14H40)},
  MRNUMBER = {3952016},
MRREVIEWER = {James\ H.\ Stankewicz},
URL = {https://msp.org/obs/2019/2-1/p16.xhtml},
}

@article {kani1994elliptic,
    AUTHOR = {Kani, E.},
     TITLE = {Elliptic curves on abelian surfaces},
   JOURNAL = {Manuscripta Math.},
  FJOURNAL = {Manuscripta Mathematica},
    VOLUME = {84},
      YEAR = {(1994)},
    NUMBER = {2},
     PAGES = {199--223},
      OPTISSN = {0025-2611,1432-1785},
   MRCLASS = {14K10 (14J25 14K20)},
  MRNUMBER = {1285957},
MRREVIEWER = {C.\ A. M. Peters},
       OPTDOI = {10.1007/BF02567454},
       URL = {https://doi.org/10.1007/BF02567454},
}

@book {watson1960integral,
    AUTHOR = {Watson, G. L.},
     TITLE = {Integral quadratic forms},
    OPTSERIES = {Cambridge Tracts in Mathematics and Mathematical Physics},
    OPTVOLUME = {No. 51},
 PUBLISHER = {Cambridge U. Press, Cambridge},
      YEAR = {(1960)},
     PAGES = {xii+143},
   MRCLASS = {10.00},
  MRNUMBER = {118704},
MRREVIEWER = {B.\ W.\ Jones},
}

@book{dicksonsbook,
  title={Studies in Number Theory},
  author={ Dickson, L.},
  year={(1957)},
  publisher={U Chicago Press, Chicago, 1930.
Reprinted by Chelsea Publ. Co., New York}
}

@book{smith,
  title={On the orders and genera of ternary quadratic forms \emph{{(1867)}}},
  author={Smith, J. H. S.},
  year={(1894)},
  publisher={In: Collect. Math. Papers vol. I, Oxford, pp. 455\textsc{\textendash}509}
}

@book{jones1950arithmetic,
  title={The Arithmetic theory of Quadratic Forms},
  author={Jones, B.},
  year={(1950)},
  publisher={Carus Math. Monogr., Wiley, New York},
  URL ={https://www.jstor.org/stable/10.4169/j.ctt5hh98f}
}

@book {mumford1970abelian,
    AUTHOR = {Mumford, D.},
     TITLE = {Abelian varieties},
    OPTSERIES = {Tata Institute of Fundamental Research Studies in Mathematics},
    OPTVOLUME = {5},
 PUBLISHER = {2nd edn. Oxford University Press, London},
      YEAR = {(1970)},
     PAGES = {viii+242},
   MRCLASS = {14.51},
  MRNUMBER = {282985},
MRREVIEWER = {James\ Milne},
}

@incollection {milne1986jacobian,
    AUTHOR = {Milne, J. S.},
     TITLE = {Jacobian varieties},
 BOOKTITLE = {Arithmetic geometry ({S}torrs, {C}onn., 1984)},
     PAGES = {167--212},
 PUBLISHER = {Springer, New York},
      YEAR = {(1986)},
      ISBN = {0-387-96311-1},
   MRCLASS = {14H40},
  MRNUMBER = {861976},
}

@book {davidcox,
    AUTHOR = {Cox, D. A.},
     TITLE = {Primes of the form {$x^2 + ny^2$}},
    SERIES = {A Wiley-Interscience Publ.},
      OPTNOTE = {Fermat, class field theory and complex multiplication},
 PUBLISHER = {John Wiley \& Sons, Inc., New York},
      YEAR = {(1989)},
     PAGES = {xiv+351},
      OPTISBN = {0-471-50654-0; 0-471-19079-9},
   MRCLASS = {11A41 (11F11 11R11 11R16 11R18 11R37 11Y11)},
  MRNUMBER = {1028322},
MRREVIEWER = {Andrew\ Bremner},
URL = {https://onlinelibrary.wiley.com/doi/book/10.1002/9781118032756},
}

@article {kaplansky,
    AUTHOR = {Kaplansky, I.},
     TITLE = {Ternary positive quadratic forms that represent all odd
              positive integers},
   JOURNAL = {Acta Arith.},
  FJOURNAL = {Acta Arithmetica},
    VOLUME = {70},
      YEAR = {(1995)},
    NUMBER = {3},
     PAGES = {209--214},
      OPTISSN = {0065-1036,1730-6264},
   MRCLASS = {11E25 (11D85)},
  MRNUMBER = {1322563},
MRREVIEWER = {Jorge\ F.\ Morales},
      OPTDOI = {10.4064/aa-70-3-209-214},
       URL = {https://doi.org/10.4064/aa-70-3-209-214},
}

@article {Jeremy_Rouse,
    AUTHOR = {Rouse, J.},
     TITLE = {Quadratic forms representing all odd positive integers},
   JOURNAL = {Amer. J. Math.},
  FJOURNAL = {American Journal of Mathematics},
    VOLUME = {136},
      YEAR = {(2014)},
    NUMBER = {6},
     PAGES = {1693--1745},
      OPTISSN = {0002-9327,1080-6377},
   MRCLASS = {11E20 (11D85 11E12 11E25 11H55)},
  MRNUMBER = {3282985},
MRREVIEWER = {Laurent\ Habsieger},
       OPTDOI = {10.1353/ajm.2014.0041},
       URL = {https://doi.org/10.1353/ajm.2014.0041},
}

@phdthesis{kir2024curvesPHD,
  author       = {Kir, H},
  title        = {Curves of genus 2 and quadratic forms},
  school       = {Queen’s University},
  year         = {2024},
  month        = {09},
  note         = {Doctoral thesis, Department of Mathematics, Statistics},
  url          = {https://qspace.library.queensu.ca/items/119efd2b-59b4-4e19-9ff0-9eea3914560e}
}

@misc{GRV,
      title={An arithmetic intersection for squares of elliptic curves with complex multiplication}, 
      author={E. L. García and C. Ritzenthaler and F. R. Villegas},
      year={(2024)},
      eprint={2412.08738},
      archivePrefix={arXiv},
      primaryClass={math.NT},
      url={https://arxiv.org/abs/2412.08738}, 
}

@misc{Eda_degree_analysis,
      author = {E. Kırımlı and G. Korpal},
      title = {Refined Humbert Invariants in Supersingular Isogeny Degree Analysis},
      howpublished = {Cryptology {ePrint} Archive, Paper 2025/1605},
      year = {2025},
      url = {https://eprint.iacr.org/2025/1605}
}

@article{Humbert,
author = {G. Humbert},
journal = {Journal de Mathématiques Pures et Appliquées},
pages = {297-401},
title = {Sur les fonctions abéliennes singulières {I} (premier Mémoire)},
volume = {5},
year = {(1899)},
}

@Manual{SageMath,
  key          = {SageMath},
  author       = {The Sage Developers},
  title        = {SageMath, the Sage Mathematics Software System (Version 10.6)},
  note         = {\url{https://www.sagemath.org}},
  year         = {(2025)},
}
\end{document}